\documentclass[12pt]{amsart}
\usepackage{epsfig} 
\usepackage{color}
\usepackage{psfrag}
\usepackage{graphicx}
\usepackage{amssymb}
\usepackage{amsmath}
\usepackage{latexsym}
\usepackage{enumerate}
\usepackage[mathscr]{eucal}
\usepackage[isolatin]{inputenc}
\usepackage{upref}
 
\makeatletter
\@namedef{subjclassname@2010}{%
 \textup{2010} Mathematics Subject Classification}
\makeatother

\newtheorem{thm}{Theorem}[section]

\theoremstyle{definition}
\newtheorem{defin}[thm]{Definition}
\newtheorem{rem}[thm]{Remark}

\numberwithin{equation}{section}

 \newcommand{\setN}{\mathbb{N}}

\begin{document}

\baselineskip=17pt

\title[]
{Further results and examples for\\
formal mathematical systems\\ 
with structural induction}

\author{Matthias Kunik}
\address{Universit\"{a}t Magdeburg\\
IAN \\
Geb\"{a}ude 02 \\
Universit\"{a}tsplatz 2 \\
D-39106 Magdeburg \\
Germany}
\email{matthias.kunik@ovgu.de}

\date{\today}
\maketitle

\begin{abstract}
In the former article ``Formal mathematical systems including
a structural induction principle'' we have presented
a unified theory for formal mathematical systems
including recursive systems closely related to formal grammars, including 
the predi\-cate calculus as well as a formal induction principle.
In this paper we present some further results and examples 
in order to illustrate how this theory works.
\end{abstract}

theory,
languages and formal grammars, structural induction, \\
$\omega$-consistency of Peano arithmetic.\\


\section{Introduction}\label{intro}

In this article I refer to my former work \cite{Ku}, which is inspired by Hofstadter's book \cite{Ho} as well as by Smullyan's  
``Theory of formal systems" in \cite{Sm}.

The recursive systems introduced in \cite[Section 1]{Ku}
may be regarded as variants of formal grammars,
but they are better adapted for use in mathe\-matical logic and enable us 
to generate in a simple way the recursively enumerable relations 
between lists of terms over a finite alphabet, using the R-axioms and the
R-rules of inference. The R-axioms of a recursive system
are special quantifier-free positive horn formulas. 
In addition, the recursive system contains R-axioms for the use of equations. 
Three R-rules of inference provide the use of R-axioms, the Modus Ponens Rule and a 
simple substitution mechanism in order to obtain conclusions from the given R-axioms.

In Section \ref{recursive} of the paper on hand we present an example 
of a recursive system which represents the natural numbers in two different ways.
The recursive system generates a specific relation 
between the dual representation of any natural number $n$
and its representation as a tally $a^n=a\ldots a$ of length $n$ with the single symbol $a$.

In \cite[Section 3]{Ku} a general recursive system $S$ is embedded into a formal mathe\-matical
system $[M;{\mathcal L}]$ based on the predicate calculus and a formal induction principle.
The set ${\mathcal L}$ of restricted argument lists 
contains the variables and is closed with respect to substitutions. The embedding is consistent in the sense that the R-axioms of $S$ 
with argument lists in ${\mathcal L}$ will become special axioms of $[M;{\mathcal L}]$ 
and that the R-rules of inference
with substitutions of lists restricted to ${\mathcal L}$ will be special
rules of inference in $[M;{\mathcal L}]$. The formal structural induction in a mathe\-matical system
is performed with respect to the axioms of the underlying recursive system $S$, and 
the formal induction principle for the natural numbers is a special case.

The three examples in Section \ref{general_valid} make use of the axioms and rules
for managing formulas with quantifiers in the formal mathematical systems, and namely 
the first example is needed in Section \ref{question}.

In Section \ref{cf} we will present some technical results
concerning the substitution of variables in formulas,
because substitutions in formulas with quantifiers need special care.

An example with formal induction involving equations will be given
in Section \ref{induction}.
The underlying recursive system of $[M;{\mathcal L}]$ is simple and similar to that in 
Section \ref{recursive}, but its R-axioms contain equations, and the formulas which we will deduce
in $[M;{\mathcal L}]$ need more effort than it seems at a first glance.

In Section \ref{prime_formulas} we develop a simple procedure in order to eliminate certain prime formulas from formal proofs which do not occur with a given arity in the basis axioms of the mathematical system.

In \cite[Section 5]{Ku} we have stated Conjecture (5.4) which characterizes the provability
of variable-free prime formulas 
in special axio\-ma\-tized mathematical systems $[M;{\mathcal L}]$ whose basis-axioms
coincide with the basis R-axioms of their underlying recursive systems.
In Section \ref{pa} of the paper on hand we present a proof of this conjecture
via Theorem \ref{mainthm}. At least under a natural interpretation of the formulas
the theorem shows that the axioms and rules of inference
including the Induction Rule (e) from \cite[(3.13)(e)]{Ku}
correspond to correct methods of deduction. As a further application of Theorem \ref{mainthm} we will give a proof for the $\omega$-consistency of the Peano arithmetic,
see Theorem \ref{pa_omega_consistency}.

In Section \ref{question} we will come back to the recursive system $S$
from Section \ref{recursive} and will present another instructive example 
for the use of the Induction Rule (e).

\section{Recursive systems} \label{recursive}
For the preparation of this section we need \cite[(1.1)-(1.12)]{Ku}. There recursively enumerable relations
are defined. These are special relations between lists of symbols, and they are generated in a very simple way by three rules of inference, namely Rules (a), (b) and (c) given in (1.11). We start with the following example:

\subsection{Dual representation of natural numbers}\label{dual_recursive}\hfill\\
We consider the recursive system $S=[A;P;B]$ with
$A=[a;0;1]$, $P=[D]$, with distinct variables $x,y \in X$ and with $B$ consisting of the following six basis R-axioms:
\begin{itemize}
\item[($\alpha$)] $D\, 1$
\item[($\beta$)] $\to ~D\, x ~ D\, x0$
\item[($\gamma$)] $\to ~D\, x ~ D\, x1$\\
\item[($\delta$)] $D\, 1,a$
\item[($\varepsilon$)] $\to ~D\, x,y ~ D\, x0,yy$
\item[($\zeta$)] $\to ~D\, x,y ~ D\, x1,yya$
\end{itemize}
\noindent
The $1$-ary predicate $D\, x$ represents natural numbers $x$ in dual form.
The $2$-ary predicate $D\, x,y$ gives the dual representation $x$ 
of a natural number $y=a^n$, represented as a tally with the single symbol ``$a$''. Note that the predicate symbol 
``$D$'' is used $1$-ary as well as $2$-ary within the same
recursive system $S$, which has to be mentioned separately in each case.
\noindent
There results another recursive 
system $\tilde{S}$ if we replace the first three R-axioms by a single one
$\to ~D\, x,y ~ D\, x$.
The elementary prime R-formulas derivable  in $S$ and $\tilde{S}$ are the same. 
In Section \ref{question} we will come back to the recursive system $S$
and will present an instructive example for a mathematical system with formal induction.

\noindent
Now we present an R-derivation of the formula
$D\,101,aaaaa$ in the recursive system $S$. It means
that 5 (represented by $a^5=aaaaa$) has the dual representation $101$:
\begin{itemize}
\item[(1)] $D\, 1,a$ \quad Rule (a) and ($\delta$).
\item[(2)] $\to ~D\, x,y ~ D\, x0,yy$ 
\quad Rule (a) and ($\varepsilon$).
\item[(3)] $\to ~D\, x,y ~ D\, x1,yya$ 
\quad Rule (a) and ($\zeta$).
\item[(4)] $\to ~D\, 1,y ~ D\, 10,yy$ 
\quad Rule (c), (2) with $x=1$.
\item[(5)] $\to ~D\, 1,a ~ D\, 10,aa$ 
\quad Rule (c), (4) with $y=a$.
\item[(6)] $D\, 10,aa$ 
\quad Rule (b), (1) and (5).
\item[(7)] $\to ~D\, 10,y ~ D\, 101,yya$ 
\quad Rule (c), (3) with $x=10$.
\item[(8)] $\to ~D\, 10,aa ~ D\, 101,aaaaa$ 
\quad Rule (c), (7) with $y=aa$.
\item[(9)] $D\, 101,aaaaa$ 
\quad Rule (b), (6) and (8).
\end{itemize}

\section{Formal mathematical systems}\label{fms}
\noindent
For the preparation of this section we need \cite[(3.1)-(3.15)]{Ku}.
In \cite[Section 3]{Ku} a recursive system $S$ is embedded into a formal mathe\-matical
system $M$. This embedding is consistent in the sense that the R-axioms of $S$ will become 
special axioms of $M$ and that the R-rules of inference will be special
rules of inference in $M$. In \cite[(3.13)]{Ku} we use five rules of inference, namely Rules (a)-(e).
Rule (e) enables formal induction with respect to the recursively enumerable relations generated by the underlying recursive system $S$.\\ 

\noindent
In \cite[(3.15)]{Ku} formal mathe\-matical systems $[M;\mathcal{L}]$
with restrictions in the argument lists 
of the formulas are introduced. The set of restricted argument lists $\mathcal{L}$ contains the variables and is closed with respect to substitutions.

\subsection{Generally valid formulas with quantifiers}\label{general_valid}\hfill\\
\noindent
{\bf Example 1:} This is needed in Section \ref{question}. Let $F$ be a formula of a mathematical system $[M;\mathcal{L}]$ and let $x\in X$ be a variable. Then we obtain the following proof of the generally valid formula $\to F ~ \exists x\, F$ in $[M;\mathcal{L}]$, 
using the rules in \cite[(3.13)]{Ku}.
\begin{enumerate}
    \item $\to~\forall x \, \neg F ~ \neg F$ \qquad
    Rule (a), quantifier axiom (3.11)(a).
    \item $\to~\,\to~\forall x\,\neg F~\neg F ~\, \to F~\,
    \neg\;\forall x\,\neg F$\\ 
    Rule (a) with the identically true propositional function\\
    $\to~\to\,\xi_1\,\neg\xi_2 ~ \to\,\xi_2\,\neg\xi_1$ \,.
    \item $\to F~\,\neg\;\forall x\,\neg F$ 
    \qquad Rule (b), (1), (2).
    \item $\leftrightarrow ~ \neg \, \forall x \neg F ~ \exists x\, F$ 
    \qquad Rule (a), quantifier axiom (3.11)(c).
    \item $\to ~\, \to F ~\,\neg\forall x\,\neg F 
    ~\, \to ~\,
    \leftrightarrow ~ \neg \, \forall x \, \neg F ~ \exists x\, F
    ~\, \to ~F ~ \exists x\, F$\\
    Rule (a) with the identically true propositional function\\
    $\to ~ \to\,\xi_1\, \xi_2 ~ \to ~
    \leftrightarrow \, \xi_2 \, \xi_3 ~ \to\,\xi_1\,\xi_3$\,.
    \item $\to ~\,
    \leftrightarrow ~ \neg \, \forall x \, \neg F ~ \exists x\, F
    ~\, \to ~F ~ \exists x\, F$\\
    Rule (b), (3), (5).
    \item $\to F ~ \exists x\, F$ 
    \qquad Rule (b), (4), (6).
\end{enumerate}

\noindent\\
{\bf Example 2:} This example plays a crucial role in the proof of 
\cite[(3.17) Theorem]{Ku}.
Let $F, G$ be formulas of a mathematical system $[M;\mathcal{L}]$
and let $x \in X$ be a variable. Then we obtain the following proof of the generally valid formula
$\to ~ \forall \, x \to F G ~~\to~\forall \, x F ~ \forall \, x G$ in $[M;\mathcal{L}]$, using the rules in \cite[(3.13)]{Ku}.
  \begin{enumerate}
    \item $\to \, \forall \, x F  ~ F$ \qquad
    Rule (a), quantifier axiom (3.11)(a).
    \item $\rightarrow ~ \forall \, x  \to F G ~\to  F G$ \qquad
    Rule (a), quantifier axiom (3.11)(a).
    \item 
		\begin{equation*}
		\begin{split}
		  \to ~&~\to \, \forall \, x F  ~ F\\
			\to ~&~\to \, \forall \, x  \to F G ~\to  F G \\
				   &~\to \, \forall \, x  \to F G ~\to  \forall \, x F \, G \\		
		\end{split}
		\end{equation*}
    Rule (a) with the identically true propositional function
		\begin{equation*}
		\begin{split}
		  \to ~~&\to ~\xi_1\,~ \xi_2\\
      \to ~~&\to ~\xi_3 ~ \to ~ \xi_2\,~ \xi_4	\\	
			&\to ~\xi_3 ~ \to ~ \xi_1\,~ \xi_4	\\	
		\end{split}
		\end{equation*}
   and $\xi_1=\forall \, x F$; $\xi_2=F$;
	$\xi_3=\forall \, x  \to F G$; $\xi_4=G$.
	\item 
		\begin{equation*}
		\begin{split}
			\to ~&~\to \, \forall \, x  \to F G ~\to  F G \\
				   &~\to \, \forall \, x  \to F G ~\to  \forall \, x F \, G \\		
		\end{split}
		\end{equation*}
  Rule (b), (1), (3).
	\item $~\to \, \forall \, x  \to F G ~\to  \forall \, x F \, G$ \qquad Rule (b), (2), (4).
  \item $\forall \, x \, \to ~ \forall \, x \to F G ~\to~\forall \, x F ~ G$ \qquad Rule (d), (5).
		\item Next we apply Rule (a), quantifier axiom (3.11)(b):
		\begin{equation*}
		\begin{split}
		 \to ~&~\forall \, x \, \to ~ \forall \, x \to F G ~\to~\forall \, x F ~ G\\
			\to ~&~\forall \, x  \to F G ~\, \forall \, x\to~  \forall \, x\,F \, G \\		
		\end{split}
		\end{equation*}
		\item $\to ~\forall \, x  \to F G ~\, \forall \, x\to~  \forall \, x\,F \, G$ \qquad
		Rule (b), (6), (7).
		\item $\to ~\forall \, x  \to  \forall \, x\,F G ~ \to~  \forall \, x\,F ~\, \forall \, x\,G$\\
		Rule (a), quantifier axiom (3.11)(b).
    \item 
		\begin{equation*}
		\begin{split}
		  \to ~&~\to ~\forall \, x  \to F G ~\, \forall \, x\to~  \forall \, x\,F \, G\\
			\to ~&~\to ~\forall \, x  \to  \forall \, x\,F G ~ \to~  \forall \, x\,F ~\, \forall \, x\,G\\
				   &~\to \, \forall \, x  \to F G ~\to  \forall \, x \, F ~\forall \, x \, G \\		
		\end{split}
		\end{equation*}
    Rule (a) with the identically true propositional function
		\begin{equation*}
		\begin{split}
		  \to ~~&\to ~\xi_1 \,\xi_2\\
      \to ~~&\to ~\xi_2 ~ \to ~ \xi_3\,~ \xi_4\\
				&\to ~\xi_1 ~ \to ~ \xi_3\,~ \xi_4\\
		\end{split}
		\end{equation*}
   and $\xi_1=\forall \, x  \to F G $; $\xi_2=\, \forall \, x\to~  \forall \, x\,F \, G$;
	$\xi_3=\forall \, x\,F$; $\xi_4=\forall \, x \,G$.
    \item 
		\begin{equation*}
		\begin{split}
			\to ~&~\to ~\forall \, x  \to  \forall \, x\,F G ~ \to~  \forall \, x\,F ~\, \forall \, x\,G\\
				   &~\to \, \forall \, x  \to F G ~\to  \forall \, x \, F ~\forall \, x \, G \\		
		\end{split}
		\end{equation*}
    Rule (b), (8), (10).
		 \item 
		$~\to \, \forall \, x  \to F G ~\to  \forall \, x \, F ~\forall \, x \, G $ \qquad
    Rule (b), (9), (11).
\end{enumerate}

The formulas in Example 1 and Example 2 are generally valid 
because we have only used non-basis axioms and the rules of inference.\\

\noindent
{\bf Example 3:} We show that $x\not\in \mbox{free}(F)$ is an essential restriction for
the quantifier axiom
~$\to ~ \forall x\to F\,G
    ~ \to F\,~\, \forall x G$~ in \cite[(3.11)(b)]{Ku}. 
Let $a\in A$ be a constant and put $F=G=~\sim x,a$ with $x\in \mbox{free}(F)$. 
Ignoring the condition $x\not\in \mbox{free}(F)$  would give the invalid ``proof''\\

\begin{enumerate}
    \item $\to ~\, \forall x\, \to \, \sim x,a \, \sim x,a
    ~\, \to ~\,\sim x,a ~~\forall x \, \sim x,a$\\ 
    incorrect use of \cite[(3.11)(b)]{Ku}.
    \item $\to \, \sim x,a \, \sim x,a$ 
    \qquad Rule (a), since $\to\xi_1\;\xi_1$ is identically true.
    \item $\forall x \, \to \, \sim x,a \, \sim x,a$ 
    \qquad Rule (d), (2).
    \item $\to ~\,\sim x,a ~~\forall x \, \sim x,a$ 
    \qquad Rule (b), (1), (3).
    \item $\to ~\,\sim a,a ~~\forall x \, \sim x,a$
    \qquad Rule (c), (4) with $x=a$. 
      \item $\sim x,x$ \qquad Rule (a), axiom of equality.
    \item $\sim a,a$ \qquad Rule (c), (6) with $x=a$.
    \item $\forall x \, \sim x,a$
    \qquad Rule (b), (5), (7). 
\end{enumerate}
In general the result $\forall x \, \sim x,a$ is false.\\

\subsection{Collision-free substitutions in formulas}\label{cf}\hfill\\
\noindent
Let $[M;\mathcal{L}]$ be a mathematical system with restricted argument lists in $\mathcal{L}$ ,
let $F$, $G$ be formulas in $[M;\mathcal{L}]$ and $x,y,z \in X$.
We make especially use of \cite[(3.1)-(3.7)]{Ku} and want to present 
the whole proof of \cite[Lemma (3.16)(a)]{Ku} with technical details.
This is necessary because substitutions in formulas with quantifiers need special care.
The lemma states that we have for all $x, z \in X$ with $z \notin \mbox{var}(F)$:
\begin{itemize}
\item[(i)] $\mbox{CF}(F;z;x)$ and
\item[(ii)] $\mbox{CF}(F\frac{z}{x};x;z)$ and
\item[(iii)] $F\frac{z}{x}\frac{x}{z} = F$\,.
\end{itemize}
\begin{proof}
We say that a formula $F$ in $[M;\mathcal{L}]$ satisfies the condition $(\star)$
if we have for all $x \in X\setminus \mbox{free}(F)$ and for all $\mu \in \mathcal{L}$:
\begin{equation*}
\mbox{CF}(F;\mu;x) \qquad \mbox{and} \qquad F\frac{\mu}{x}= F\,.
\end{equation*}
We first use induction on well-formed formulas to show that condition $(\star)$
is satisfied for all formulas $F$ in $[M;\mathcal{L}]$. Afterwards we prove
\cite[Lemma (3.16)(a)]{Ku}. 
\begin{itemize}
\item[(a)] From \cite[(3.7)(a)]{Ku} and \cite[(3.5)]{Ku}, 
\cite[(3.6)(a)]{Ku} we see that 
condition $(\star)$ is satisfied for all prime formulas $F$.
\item[(b)]  If $(\star)$ is satisfied for a formula $F$,
then also for the formula $\neg F$ due to \cite[(3.7)(b)]{Ku} and 
\cite[(3.6)(b)]{Ku}.
\item[(c)]  Next we assume that $F$ and $G$ both satisfy condition $(\star)$ 
(induction hypothesis) and put
$H=J \, F G$ for $J \in \{\to\,;\,\leftrightarrow\,;\,\&\,;\,\vee\}$.
Let $x \in X \setminus \mbox{free}(H)$, $\mu \in \mathcal{L}$.
From $\mbox{free}(H)= \mbox{free}(F) \cup \mbox{free}(G)$ we obtain that
$x \in X \setminus \mbox{free}(F)$ as well as $x \in X \setminus \mbox{free}(G)$.
We conclude from our induction hypothesis and \cite[(3.7)(c)]{Ku} that
$\mbox{CF}(F;\mu;x)$, $\mbox{CF}(G;\mu;x)$ and $\mbox{CF}(H;\mu;x)$.
Next we obtain from our induction hypothesis and \cite[(3.6)(c)]{Ku} that
\begin{equation*}
\begin{split}
\mbox{SbF}(H;\mu;x)&=\mbox{SbF}(J \, FG;\mu;x)\\
&= J \, \mbox{SbF}(F;\mu;x) \, \mbox{SbF}(G;\mu;x)\\ 
&= J \, F G = H\,.\\
\end{split}
\end{equation*}
\item[(d)]  Assume that $F$ satisfies condition $(\star)$ (induction hypothesis) 
and put $H=Q y \, F$ for $Q \in \{\forall\,;\,\exists\}$.
Let $x \in X \setminus \mbox{free}(H)$, $\mu \in \mathcal{L}$
and note that $$\mbox{free}(H)= \mbox{free}(F) \setminus \{y\}\,.$$
We have $\mbox{CF}(H;\mu;x)$ immediately from \cite[(3.7)(d)]{Ku}
and obtain that $x \in X \setminus \mbox{free}(F)$ or $x=y$.
For the substitution we distinguish two cases 
according to \cite[(3.6)(d)]{Ku}:

\noindent
Case 1: $x=y$. Then
$$
\mbox{SbF}(H;\mu;x)=\mbox{SbF}(Q y \, F;\mu;x)=Q y \, F=H\,.
$$
Case 2: $x \neq y$ and $x \in X \setminus \mbox{free}(F)$. Then
we obtain from our induction hypothesis that
$$
\mbox{SbF}(H;\mu;x)=\mbox{SbF}(Q y \, F;\mu;x)=Q y \, \mbox{SbF}(F;\mu;x)
= Q y \, F=H\,.
$$
\end{itemize}
We have shown that condition $(\star)$ is valid for \textit{all} 
formulas $F$ in $[M;\mathcal{L}]$. \\

\noindent
We say that a formula $F$ in $[M;\mathcal{L}]$ satisfies the condition $(\star \star)$
if it satisfies (i), (ii) and (iii) for all $x, z \in X$ with $z \notin \mbox{var}(F)$.
For the proof of \cite[Lemma (3.16)(a)]{Ku} we use induction on well-formed formulas to show that condition $(\star \star)$ is satisfied for all formulas $F$ in $[M;\mathcal{L}]$. 

\begin{itemize}
\item[(a)] From \cite[(3.7)(a)]{Ku} and \cite[(3.5)]{Ku}, 
\cite[(3.6)(a)]{Ku} we see that 
condition $(\star \star)$ is satisfied for all prime formulas $F$.
\item[(b)]  If $(\star \star)$ is satisfied for a formula $F$,
then also for the formula $\neg F$ due to \cite[(3.7)(b)]{Ku} 
and \cite[(3.6)(b)]{Ku}.
\item[(c)]  Assume that $F$ and $G$ both satisfy condition $(\star \star)$ 
(induction hypothesis) and put $H = J\,FG$ for $J \in \{\to\,;\,\leftrightarrow\,;\,\&\,;\,\vee\}$.
Let $x, z \in X$ with $z \notin \mbox{var}(H)=\mbox{var}(F) \cup \mbox{var}(G)$.
We have
$$z \notin \mbox{var}(F)\,, \quad z \notin \mbox{var}(G)$$
and conclude from our induction hypothesis and 
\cite[(3.7)(c)]{Ku}, \cite[(3.6)(c)]{Ku} that
\begin{itemize}
\item[(i)$'$] $\mbox{CF}(F;z;x)$, $\mbox{CF}(G;z;x)$ and hence $\mbox{CF}(H;z;x)$\,,
\item[(ii)$'$] $\mbox{CF}(F\frac{z}{x};x;z)$, $\mbox{CF}(G\frac{z}{x};x;z)$ 
and hence $\mbox{CF}(H\frac{z}{x};x;z)$\,,
\item[(iii)$'$] $H\frac{z}{x}\frac{x}{z}= 
J\,\mbox{SbF}(F\frac{z}{x};x;z)\,\mbox{SbF}(G\frac{z}{x};x;z)=J\,FG=H\,.$
\end{itemize}
\item[(d)]  Assume that $F$ satisfies condition $(\star \star)$ (induction hypothesis) 
and put $H=Q y \, F$ for $Q \in \{\forall\,;\,\exists\}$.
Let $x, z \in X$, $z \notin \mbox{var}(H)$ and
note that $z \notin \mbox{var}(F) \cup \{y\}$\,, especially $z \neq y$.

\noindent
Case 1: We suppose that $x \notin \mbox{free}(H)$\,.
Since $H$ satisfies the former condition $(\star)$ we obtain that
$\mbox{CF}(H;z;x)$, $H\frac{z}{x}=H$, and from $z \notin \mbox{var}(H)$ that
$\mbox{CF}(H\frac{z}{x};x;z)$ as well as $H\frac{z}{x}\frac{x}{z}= H\frac{x}{z}=H$.
Hence $H$ satisfies condition $(\star\star)$ in case 1.

\noindent
Case 2: We suppose that $x \in \mbox{free}(H)=\mbox{free}(F)\setminus \{y\}$\,.
We have $\mbox{CF}(F;z;x)$ from our induction hypothesis, recall that 
$y \neq z$ and conclude $\mbox{CF}(H;z;x)$ from \cite[(3.7)(d)ii)]{Ku}.
Next we use that $y \neq x$ and have from the induction hypothesis
$$H\frac{z}{x}=Q y\,F\frac{z}{x} \quad \mbox{and} \quad  \mbox{CF}(F\frac{z}{x};x;z)\,.$$
We obtain $\mbox{CF}(H\frac{z}{x};x;z)$ from \cite[(3.7)(d)ii)]{Ku} and see
$$  
\mbox{SbF}(H\frac{z}{x}; x; z) =Q y\,\mbox{SbF}(F\frac{z}{x};x;z)=Q y\,F=H
$$
from $z \neq y$ and the induction hypothesis.
Hence $H$ satisfies condition $(\star\star)$ in case 2.
\end{itemize}
We have shown that condition $(\star \star)$ is valid for \textit{all} 
formulas $F$ in $[M;\mathcal{L}]$. 
\end{proof}

\subsection{An example with formal induction and equations}
\label{induction}\hfill\\
\noindent
With 
$A_S := [\,a\,;\,b\,; \, f \,]$ and
$P_S := [\,W\,]$ 
we define a recursive system $S=[A_S;P_S;B_S]$ by the following list $B_S$
of basis R-axioms, where $x, y,s,t,u,v \in X$ are distinct variables:\\

\begin{tabular}{llll}
(1)  & $W\,a$ &&\\ 
(2)  & $W\,b$ &&\\
(3)  & $\to ~ W\,x ~ \to ~ W\,y\,~ W\,xy$ &&\\
(4)  & $\sim\,f(a),a$ &&\\  
(5)  & $\sim\,f(b),b$ &&\\
(6)  & $\to ~ W\,x ~ \to ~ W\,y\,~ \sim\,f(xy),f(y)f(x)$\,. &&\\
\end{tabular}

\noindent\\
The strings consisting of the symbols $a$ and $b$ are 
generated by the R-axioms (1)-(3). They are indicated by the
predicate symbol $W$, which is used only 1-ary here,
whereas $f$ denotes the operation which reverses the order
of such a string. For example, $\sim\,f(abaab),baaba$ is R-derivable, 
and equations like $\sim\,f(abaab),f(aab)ba$ and R-formulas
like \, $W\,f(aab)ba$ are also R-derivable.
The R-formula
\begin{equation}\label{reverse}
\to \,W x~\sim\,f(f(x)),x \tag{$\star$}
\end{equation} 
is not R-derivable in $S$.
But we will show that the latter formula is provable in the
mathematical system $[M;\mathcal{L}]$
with $M=[S;A_S;P_S;B_S]$ and the set $\mathcal{L}$
generated by the following rules:
\begin{itemize}
\item[(i)] $x \in \mathcal{L}$ for all $x \in X$,
\item[(ii)] $a \in \mathcal{L}$ and $b \in \mathcal{L}$,
\item[(iii)] If $\lambda, \mu \in \mathcal{L}$ then 
$\lambda \mu \in \mathcal{L}$,
\item[(iv)] If $\lambda \in \mathcal{L}$ then $f(\lambda) \in \mathcal{L}$.
\end{itemize}
The R-axioms (1)-(6) also form a proof in the mathematical system 
$[M;\mathcal{L}]$ which is extended to the following
proof in $[M;\mathcal{L}]$:
\begin{itemize}
    \item[(7)] $\sim x,x$ \qquad Rule (a), axiom of equality.
		\item[(8)] $\sim f(s),f(s)$ \qquad Rule (c), (7) 
		with $x=f(s)$. 
    \item[(9)] $\to~\sim f(s),f(s) ~\to~\sim s,t ~\sim f(s),f(t)$\\
    \qquad Rule (a), axiom of equality.
    \item[(10)] $\to~\sim s,t ~\sim f(s),f(t)$
    \qquad Rule (b), (8), (9).
    \item[(11)] $\to~\sim f(a),t ~\sim f(f(a)),f(t)$
    \qquad Rule (c), (10) with $s=f(a)$.
	  \item[(12)] $\to~\sim f(a),a ~\sim f(f(a)),f(a)$
    \qquad Rule (c), (11) with $t=a$.
		\item[(13)] $\sim f(f(a)),f(a)$ \qquad Rule (b), (4), (12).
	  \item[(14)] $\to~\sim s,t~\to~\sim t,u~\sim s,u$\\
    \qquad Rule (a), axiom of equality.
	  \item[(15)] $\to~\sim f(f(a)),t~\to~\sim t,u~\sim f(f(a)),u$\\
    \qquad Rule (c), (14) with $s=f(f(a))$.
	  \item[(16)] $\to~\sim f(f(a)),f(a)~\to~\sim f(a),u~\sim f(f(a)),u$\\
    \qquad Rule (c), (15) with $t=f(a)$.
	  \item[(17)] $\to~\sim f(f(a)),f(a)~\to~\sim f(a),a~\sim f(f(a)),a$\\
    \qquad Rule (c), (16) with $u=a$.
	  \item[(18)] $\to~\sim f(a),a~\sim f(f(a)),a$ 
		\qquad Rule (b), (13), (17).
	  \item[(19)] $\sim f(f(a)),a$ \qquad Rule (b), (4), (18).
		\item[(20)] $\to~\sim f(b),t ~\sim f(f(b)),f(t)$
    \qquad Rule (c), (10) with $s=f(b)$.
	  \item[(21)] $\to~\sim f(b),b ~\sim f(f(b)),f(b)$
    \qquad Rule (c), (20) with $t=b$.
		\item[(22)] $\sim f(f(b)),f(b)$ \qquad Rule (b), (5), (21).
	  \item[(23)] $\to~\sim f(f(b)),t~\to~\sim t,u~\sim f(f(b)),u$\\
    \qquad Rule (c), (14) with $s=f(f(b))$.
	  \item[(24)] $\to~\sim f(f(b)),f(b)~\to~\sim f(b),u~\sim f(f(b)),u$\\
    \qquad Rule (c), (23) with $t=f(b)$.
	  \item[(25)] $\to~\sim f(f(b)),f(b)~\to~\sim f(b),b~\sim f(f(b)),b$\\
    \qquad Rule (c), (24) with $u=b$.
	  \item[(26)] $\to~\sim f(b),b~\sim f(f(b)),b$ 
		\qquad Rule (b), (22), (25).
	  \item[(27)] $\sim f(f(b)),b$ \qquad Rule (b), (5), (26).
	  \item[(28)] $\to ~\sim s,s ~ \to ~\sim s,t ~\sim t,s$ 
		\qquad Rule (a), axiom of equality.
	  \item[(29)] $\sim s,s$ \qquad Rule (a), axiom of equality.
	  \item[(30)] $~ \to ~\sim s,t ~\sim t,s$ \qquad Rule (b), (28), (29).
	  \item[(31)] $~ \to ~\sim f(a),t ~\sim t,f(a)$ 
		\qquad Rule (c), (30) with $s=f(a)$.
	  \item[(32)] $~ \to ~\sim f(a),a ~\sim a,f(a)$ 
		\qquad Rule (c), (31) with $t=a$.
	  \item[(33)] $\sim a,f(a)$ 
		\qquad Rule (b), (4), (32).
	  \item[(34)] $~ \to ~\sim f(b),t ~\sim t,f(b)$ 
		\qquad Rule (c), (30) with $s=f(b)$.
	  \item[(35)] $~ \to ~\sim f(b),b ~\sim b,f(b)$ 
		\qquad Rule (c), (34) with $t=b$.
	  \item[(36)] $\sim b,f(b)$ 
		\qquad Rule (b), (5), (35).
	  \item[(37)] $\to ~\sim s,t ~ \to ~W s ~W t$ 
		\qquad Rule (a), axiom of equality.
	  \item[(38)] $\to ~\sim a,t ~ \to ~W a ~W t$ 
		\qquad Rule (c), (37) with $s=a$.
	  \item[(39)] $\to ~\sim a,f(a) ~ \to ~W a ~W f(a)$ 
		\qquad Rule (c), (38) with $t=f(a)$.
	  \item[(40)] $\to ~W a ~W f(a)$ 
		\qquad Rule (b), (33), (39).
	  \item[(41)] $W f(a)$ \qquad Rule (b), (1), (40).
	  \item[(42)] $\to ~\sim b,t ~ \to ~W b ~W t$ 
		\qquad Rule (c), (37) with $s=b$.
	  \item[(43)] $\to ~\sim b,f(b) ~ \to ~W b ~W f(b)$ 
		\qquad Rule (c), (42) with $t=f(b)$.
	  \item[(44)] $\to ~W b ~W f(b)$ 
		\qquad Rule (b), (36), (43).
	  \item[(45)] $W f(b)$ \qquad Rule (b), (2), (44).
		\item[(46)] $\to W a \to W f(a) \to ~ \sim f(f(a)),a 
		   ~\& ~\& \, W a ~ W f(a) ~\sim f(f(a)),a$\\ 
		Rule (a), axiom of the propositional calculus.
		\item[(47)] $\to W f(a) \to ~ \sim f(f(a)),a 
		~\& ~\& \, W a ~ W f(a) ~\sim f(f(a)),a$\\ 
		Rule (b), (1), (46).
		\item[(48)] $\to ~ \sim f(f(a)),a ~ \& ~\& \, W a ~ W f(a) ~
		\sim f(f(a)),a$\\ Rule (b), (41), (47).
		\item[(49)] $\& ~\& \, W a ~ W f(a) ~
		\sim f(f(a)),a$ \qquad Rule (b), (19), (48).
		\item[(50)] $\to W b \to W f(b) \to ~ \sim f(f(b)),b 
		~\& ~\& \, W b ~ W f(b) ~
		\sim f(f(b)),b$\\ 
		Rule (a), axiom of the propositional calculus.
		\item[(51)] $\to W f(b) \to ~ \sim f(f(b)),b 
		~\& ~\& \, W b ~ W f(b) ~\sim f(f(b)),b$\\ 
		Rule (b), (2), (50).
		\item[(52)] $\to ~ \sim f(f(b)),b ~ \& ~\& \, W b ~ W f(b) ~
		\sim f(f(b)),b$\\ Rule (b), (45), (51).
		\item[(53)] $\& ~\& \, W b ~ W f(b) ~
		\sim f(f(b)),b$ \qquad Rule (b), (27), (52).
		\end{itemize}
		
		\noindent
		At this place we stop the proof in the mathematical system 
		$[M;\mathcal{L}]$, introduce two different 
		and new constant symbols $c$, $d$
		not occurring in $[M;\mathcal{L}]$ and define the extension
		$A := [\,a\,;\,b\,; \, f \,;c\,;d\,]$ of the alphabet $A_S$.
		With $M_A:=[S;\,A\,;\,P_S\,;B_S\,]$ and
		\[ {\mathcal L}' := 
		\{\, \lambda \frac{t_1}{x_1}...\frac{t_m}{x_m} \, : \,
		\lambda \in \mathcal{L},\, x_1,\ldots,x_m \in X,\,
		t_1,\ldots,t_m \in \{c,d\}\,,\,
		m \geq 0 \,\} \]
		there results the mathematical system $[M_A;\mathcal{L}']$ 
		due to \cite[Definition (4.2)(d)]{Ku} and \cite[Corollary (4.9)(a)]{Ku}.
		Next we make use of the abbreviation
		\begin{equation*}
		G(\lambda) := \& ~\& \, W \lambda ~ W f(\lambda) ~
		\sim f(f(\lambda)),\lambda \quad \mbox{with~}
		\lambda \in {\mathcal L}'
		\end{equation*}
    and adjoin to $[M_A;\mathcal{L}']$ the two statements
		\begin{equation}\label{star2}
		  \varphi_1 := G(c)\,, \quad
			\varphi_2 := G(d)\,.   \tag{$\star \star$}
		\end{equation}
    There results the extended 
		mathematical system $[M';\mathcal{L}']$ with 
		$$M' := M_A(\{\varphi_1,\varphi_2\})
		=[S;\,A\,;P_S\,;\,B_S \cup \{\varphi_1,\varphi_2\}\,]$$
		due to \cite[Definition (4.2)(b)]{Ku}. Note that the
		abbreviations $G(\lambda)$, $\varphi_1$, $\varphi_2$ 
		are not part of the formal system.
    We keep in mind that any proof in $[M;\mathcal{L}]$ 
		is also a proof in $[M';\mathcal{L}']$
		and that the mathematical systems $[M;\mathcal{L}]$,
		$[M_A;\mathcal{L}']$ and $[M';\mathcal{L}']$ 
		all have the same underlying recursive system $S$.
		Hence (1)-(53) also constitutes a proof in $[M';\mathcal{L}']$,
		and we extend it to the following proof of the formula
		$G(cd)$ in $[M';\mathcal{L}']$:	
		\begin{itemize}
		\item[(54)] $G(c)$\qquad Rule (a) with axiom $\varphi_1=G(c)$.
		\item[(55)] $G(d)$\qquad Rule (a) with axiom $\varphi_2=G(d)$.
		\item[(56)] $\to G(c)~ W\,c$\\
		Rule (a), axiom of the propositional calculus.
		\item[(57)] $\to G(c)~ W\,f(c)$\\
		Rule (a), axiom of the propositional calculus.
		\item[(58)] $\to G(c)~ \sim f(f(c)),c$\\
		Rule (a), axiom of the propositional calculus.
		\item[(59)] $\to G(d)~ W\,d$\\ 
		Rule (a), axiom of the propositional calculus.
		\item[(60)] $\to G(d)~ W\,f(d)$\\
		Rule (a), axiom of the propositional calculus.
		\item[(61)] $\to G(d)~ \sim f(f(d)),d$\\
		Rule (a), axiom of the propositional calculus.
    \item[(62)] $W\,c$\qquad Rule (b), (54), (56).
    \item[(63)] $W\,d$\qquad Rule (b), (55), (59).
    \item[(64)] $W\,f(c)$\qquad Rule (b), (54), (57).
    \item[(65)] $W\,f(d)$\qquad Rule (b), (55), (60).
    \item[(66)] $\sim f(f(c)),c$\qquad Rule (b), (54), (58).
    \item[(67)] $\sim f(f(d)),d$\qquad Rule (b), (55), (61).	
		\item[(68)] $\to ~ W\,c ~ \to ~ W\,y\,~ W\,cy$ 
		\qquad Rule (c), (3) with $x=c$.
		\item[(69)] $\to ~ W\,c ~ \to ~ W\,d\,~ W\,cd$ 
		\qquad Rule (c), (68) with $y=d$.
		\item[(70)] $\to ~ W\,d\,~ W\,cd$ \qquad Rule (b), (62), (69).
		\item[(71)] $W\,cd$ \qquad Rule (b), (63), (70).
		\item[(72)] $\to ~ W\,f(d) ~ \to ~ W\,y\,~ W\,f(d)y$ 
		\qquad Rule (c), (3) with $x=f(d)$.
		\item[(73)] $\to ~ W\,f(d) ~ \to ~ W\,f(c)\,~ W\,f(d)f(c)$\\ 
		Rule (c), (72) with $y=f(c)$.
		\item[(74)] $\to ~ W\,f(c)\,~ W\,f(d)f(c)$ \qquad Rule (b), (65), (73).
		\item[(75)] $W\,f(d)f(c)$ \qquad Rule (b), (64), (74).
		\item[(76)] $\to ~ W\,c ~ \to ~ W\,y\,~ \sim\,f(cy),f(y)f(c)$\\
		Rule (c), (6) with $x=c$.
		\item[(77)] $\to ~ W\,c ~ \to ~ W\,d\,~ \sim\,f(cd),f(d)f(c)$\\ 
		Rule (c), (76) with $y=d$.
		\item[(78)] $\to ~ W\,d\,~ \sim\,f(cd),f(d)f(c)$ 
		\qquad Rule (b), (62), (77).
		\item[(79)] $\sim\,f(cd),f(d)f(c)$ \qquad Rule (b), (63), (78).
	  \item[(80)] $\to ~\sim f(cd),t ~\sim t,f(cd)$\\
		Rule (c), (30) with $s=f(cd)$.
	  \item[(81)] $~ \to ~\sim f(cd),f(d)f(c) ~\sim f(d)f(c),f(cd)$\\
		Rule (c), (80) with $t=f(d)f(c)$.
	  \item[(82)] $\sim f(d)f(c),f(cd)$ \qquad Rule (b), (79), (81).
	  \item[(83)] $\to~ \sim f(d)f(c),t \to~ W\,f(d)f(c)~W\,t$\\
		\qquad Rule (c), (37) with $s=f(d)f(c)$.
		\item[(84)] $\to~ \sim f(d)f(c),f(cd) \to~ W\,f(d)f(c)~W\,f(cd)$\\
		\qquad Rule (c), (83) with $t=f(cd)$.
	  \item[(85)] $\to~ W\,f(d)f(c)~W\,f(cd)$ \qquad Rule (b), (82), (84).
	  \item[(86)] $W\,f(cd)$ \qquad Rule (b), (75), (85).
		\item[(87)] $\to ~\sim st,st ~ \to ~ \sim t,v\,~ 
		\sim\,st,sv$ \qquad Rule (a), axiom of equality.
		\item[(88)] $\sim\,st,st$ \qquad Rule (c), (7) with $x=st$.
		\item[(89)] $\to ~ \sim t,v\,~ 
		\sim\,st,sv$ \qquad Rule (b), (87), (88).		
		\item[(90)] $\to ~\sim st,sv ~ \to ~ \sim s,u\,~ 
		\sim\,st,uv$\qquad Rule (a), axiom of equality.
		\item[(91)]		
		\begin{equation*}
		\begin{split}
		  \to ~~&\to ~\sim t,v\,~ \sim\,st,sv\\
      \to ~~&\to ~\sim st,sv ~ \to ~ \sim s,u\,~ \sim\,st,uv	\\	
			&\to ~\sim s,u ~ \to ~ \sim t,v\,~ \sim\,st,uv	\\	
		\end{split}
		\end{equation*}
    Rule (a) with the identically true propositional function
		\begin{equation*}
		\begin{split}
		  \to ~~&\to ~\xi_1\,~ \xi_2\\
      \to ~~&\to ~\xi_2 ~ \to ~ \xi_3\,~ \xi_4	\\	
			&\to ~\xi_3 ~ \to ~ \xi_1\,~ \xi_4	\\	
		\end{split}
		\end{equation*}
   and $\xi_1=~\sim t,v$; $\xi_2=~\sim st,sv$;
	$\xi_3=~\sim s,u$; $\xi_4=~\sim st,uv$.
	\item[(92)]		
		\begin{equation*}
		\begin{split}
      \to ~~&\to ~\sim st,sv ~ \to ~ \sim s,u\,~ \sim\,st,uv	\\	
			&\to ~\sim s,u ~ \to ~ \sim t,v\,~ \sim\,st,uv	\\	
		\end{split}
		\end{equation*}
    Rule (b), (89), (91).
		\item[(93)]$\to ~\sim s,u ~ \to ~ \sim t,v\,~ \sim\,st,uv$
		\qquad	Rule (b), (90), (92).
	  \item[(94)] $\to ~ W\,f(d) ~ \to ~ W\,y\,~ 
		\sim\,f(f(d)y),f(y)f(f(d))$\\
    Rule (c), (6) with $x=f(d)$.
	  \item[(95)] $\to ~ W\,f(d) ~ \to ~ W\,f(c)\,~ 
		\sim\,f(f(d)f(c)),f(f(c))f(f(d))$\\
    Rule (c), (94) with $y=f(c)$.
		\item[(96)] $\to ~ W\,f(c)\,~ \sim\,f(f(d)f(c)),f(f(c))f(f(d))$\\
    Rule (b), (65), (95).
		\item[(97)] $\sim\,f(f(d)f(c)),f(f(c))f(f(d))$\\
    Rule (b), (64), (96).
		\item[(98)]$\to ~\sim f(f(c)),u ~ \to ~ \sim t,v\,~ \sim\,f(f(c))t,uv$\\
		Rule (c), (93) with $s=f(f(c))$.
		\item[(99)]$\to ~\sim f(f(c)),c ~ \to ~ \sim t,v\,~ \sim\,f(f(c))t,cv$\\
		Rule (c), (98) with $u=c$.
		\item[(100)]$\to ~\sim f(f(c)),c ~ \to ~ \sim f(f(d)),v\,~ \sim\,f(f(c))f(f(d)),cv$\\
		Rule (c), (99) with $t=f(f(d))$.
		\item[(101)]$\to ~\sim f(f(c)),c ~ \to ~ \sim f(f(d)),d\,~ \sim\,f(f(c))f(f(d)),cd$\\
		Rule (c), (100) with $v=d$.
		\item[(102)]$\to ~ \sim f(f(d)),d\,~ \sim\,f(f(c))f(f(d)),cd$\\
		Rule (b), (66), (101).
		\item[(103)]$\sim\,f(f(c))f(f(d)),cd$\\
		Rule (b), (67), (102).
	  \item[(104)] $\to~\sim f(f(d)f(c)),t~\to~\sim t,u~\sim f(f(d)f(c)),u$\\
    \qquad Rule (c), (14) with $s=f(f(d)f(c))$.
	  \item[(105)] $\to~\sim f(f(d)f(c)),f(f(c))f(f(d))~\to~\sim f(f(c))f(f(d)),u\\
		\sim f(f(d)f(c)),u$ \qquad Rule (c), (104) with $t=f(f(c))f(f(d))$.
	  \item[(106)] $\to~\sim f(f(d)f(c)),f(f(c))f(f(d))~\to~\sim f(f(c))f(f(d)),cd\\
		\sim f(f(d)f(c)),cd$ \qquad Rule (c), (105) with $u=cd$.
	  \item[(107)] $\to~\sim f(f(c))f(f(d)),cd~\sim f(f(d)f(c)),cd$\\
    \qquad Rule (b), (97), (106).
	  \item[(108)] $\sim f(f(d)f(c)),cd$\qquad Rule (b),(103), (107).
		\item[(109)] $\to~\sim f(cd),t ~\sim f(f(cd)),f(t)$\\
    Rule (c), (10) with $s=f(cd)$.
		\item[(110)] $\to~\sim f(cd),f(d)f(c) ~\sim f(f(cd)),f(f(d)f(c))$\\
    Rule (c), (109) with $t=f(d)f(c)$.
		\item[(111)] $\sim f(f(cd)),f(f(d)f(c))$ \qquad Rule (b), (79), (110).
	  \item[(112)] $\to~\sim f(f(cd)),t~\to~\sim t,u~\sim f(f(cd)),u$\\
	  Rule (c), (14) with $s=f(f(cd))$.
		\item[(113)] $\to~\sim f(f(cd)),f(f(d)f(c))~\to~\sim f(f(d)f(c)),u~\sim f(f(cd)),u$\\
	   Rule (c), (112) with $t=f(f(d)f(c))$.
		\item[(114)] $\to~\sim f(f(cd)),f(f(d)f(c))~\to~\sim f(f(d)f(c)),cd~\sim f(f(cd)),cd$\\
	   Rule (c), (113) with $u=cd$.
		\item[(115)] $\to~\sim f(f(d)f(c)),cd~\sim f(f(cd)),cd$\\
	   Rule (b), (111), (114).
		\item[(116)] $\sim f(f(cd)),cd$ \qquad Rule (b), (108), (115).
	  \item[(117)] $\to W cd \to W f(cd) \to ~ \sim f(f(cd)),cd 
		~G(cd)$\\ 
		Rule (a), axiom of the propositional calculus.
		\item[(118)] $\to W f(cd) \to ~ \sim f(f(cd)),cd ~G(cd)$\qquad Rule (b), (71), (117).
		\item[(119)] $\to ~ \sim f(f(cd)),cd ~ G(cd)$\\ Rule (b), (86), (118).
		\item[(120)] $G(cd)$ \qquad Rule (b), (116), (119).
\end{itemize}
We have deduced $G(cd)$ in $[M';{\mathcal L}']$. It follows from the Deduction Theorem \cite[(4.5)]{Ku} that the formula
$
\to ~ G(c) ~ \to G(d) ~ G(cd)
$
is provable in $[M_A;{\mathcal L}']$. From the generalization of the 
constant symbols $c$, $d$ according to \cite[Corollary (4.9)(b)]{Ku} 
we see that
$$
\to ~ G(x) ~ \to G(y) ~ G(xy)
$$
is provable in the original mathematical system $[M;{\mathcal L}]$.
Moreover, the formulas $G(a)$ in (49) and $G(b)$ in (53) are also provable in $[M;{\mathcal L}]$.
We apply Rule (e) in $[M;{\mathcal L}]$ on the last three formulas
and finally conclude that the formulas
$
\to ~ W\,u ~ G(u)
$
and hence
$$
\to ~ W\,x ~ W\,f(x) \quad \mbox{and} \quad  \to ~ W\,x ~ \sim \,f(f(x)),x 
$$
are provable in $[M;{\mathcal L}]$.\\

\subsection{On prime formulas not occurring in the basis axioms}\label{prime_formulas}\hfill\\
In this note we determine a simple procedure in order to eliminate prime formulas from formal proofs 
which do not occur with a given arity in the basis axioms of a mathematical system.

Let $[M;{\mathcal L}]$ with $M = [S; A_M; P_M; B_M]$ be a mathematical system with restricted argument lists in $\mathcal L$.
Assume that $q \in P_M$ does not occur $j$-ary in the basis axioms $B_M$, where $j \geq 0$ is a given integer number.
Let $[\Lambda]=[F_1;\ldots;F_l]$ be a proof in $[M;{\mathcal L}]$ with the steps $F_1 ; \ldots; F_l$.
For a variable $z \in X$ not involved in $B_S$ we put as abbreviation the contradiction 
$$ C = \& \, \forall z \, \sim z,z ~\neg \forall z \, \sim z,z \,.$$
If replace in each formula $F$ of $[M;{\mathcal L}]$ all subformulas
of the form $q \lambda_1,\ldots,\lambda_j$ with $\lambda_1,\ldots,\lambda_j \in {\mathcal L}$
by the contradiction $C$ then we obtain the formula ${\mathcal C}(F)$ with argument lists in ${\mathcal L}$.
We will show that
$$[{\mathcal C}({\Lambda})]=[{\mathcal C}(F_1);\ldots;{\mathcal C}(F_l)]$$ 
is again a proof in $[M;{\mathcal L}]$, where $q$ does not occur $j$-ary in $[{\mathcal C}({\Lambda})]$\,.
We can subsequently apply this procedure and obtain the following result:
Apart from the equations we can replace all prime formulas in the original proof $[\Lambda]$ by $C$
which do not appear as subformulas with a given arity in the basis axioms of $[M;{\mathcal L}]$.
All other prime formulas which occur as subformulas in the steps of $[\Lambda]$ are not affected
by this procedure.

In the sequel we fix the quantities $q \in P_M$ and $j \in \setN_0$
in the definition of $C$ and ${\mathcal C}(\cdot)$.

\noindent
\textit{Lemma:} Let $F$ be a formula in $[M;{\mathcal L}]$ 
Then for every list $\mu \in {\mathcal L}$ and for all variables $x \in X$ with 
$CF(F; \mu; x)$ there holds the condition $\mbox{CF}({\mathcal C}(F);\,\mu;\,x)$ 
and the equation
\[
{\mathcal C}(\mbox{SbF}(F; \mu; x)) =\mbox{SbF}({\mathcal C}(F); \mu; x)\,.
\]
\begin{proof}
\noindent
We use induction with respect to the rules for generating 
formulas in $[M;{\mathcal L}]$ and fix
a variable $x \in X $ as well as a list $\mu \in {\mathcal L}$.

\noindent
We say that a formula $F$ in $[M;{\mathcal L}]$ satisfies Condition $(*)$
if the condition $\mbox{CF}(F;\,\mu;\,x)$ implies
the condition $\mbox{CF}({\mathcal C}(F);\,\mu;\,x)$
and the equation
${\mathcal C}(\mbox{SbF}(F; \mu; x)) =\mbox{SbF}({\mathcal C}(F); \mu; x)$\,.

\noindent
We prove that Condition $(*)$ is satisfied for all formulas $F$ in $[M;{\mathcal L}]$. 
We use the definitions \cite[(3.6) and (3.7)]{Ku} and the notations occurring there 
by treating the corresponding cases (a)-(d) in these definitions. 
\begin{itemize}
\item[(a)] If $F$ is a prime formula in $[M;{\mathcal L}]$ of the form
$q \lambda_1,\ldots,\lambda_j$ then ${\mathcal C}(F)=C$ with
$\mbox{CF}(F;\,\mu;\,x)$ and $\mbox{CF}(C;\,\mu;\,x)$,
and we have
$${\mathcal C}(\mbox{SbF}(F; \mu; x)) =C=\mbox{SbF}({\mathcal C}(F); \mu; x)$$
Otherwise $F$ is a prime formula in $[M;{\mathcal L}]$ different 
from $q \lambda_1,\ldots,\lambda_j$ with ${\mathcal C}(F)=F$.
In both cases we have confirmed Condition $(*)$ for the prime formulas.
\item[(b)] We assume that Condition $(*)$ is satisfied 
for a formula $F$ in $[M;{\mathcal L}]$ and that
$\mbox{CF}(\neg \, F; \mu; x)$. Then there holds
the condition $\mbox{CF}(F; \mu; x)$, and we have 
${\mathcal C}(\neg \, F)=\neg \,{\mathcal C}(F) $. 
Since $F$ satisfies Condition $(*)$,
we conclude that $\mbox{CF}({\mathcal C}(F); \mu; x)$ and
$\mbox{CF}({\mathcal C}(\neg \, F); \mu; x)$ are valid and that
the equations
\[
{\mathcal C}(\mbox{SbF}(\neg \, F; \mu; x)) =
\neg \,{\mathcal C}(\mbox{SbF}(F; \mu; x)) =
\mbox{SbF}({\mathcal C}(\neg \, F); \mu; x)
\]
are satisfied. Thus we have confirmed Condition $(*)$ for $\neg \, F$.
\item[(c)] We assume that Condition $(*)$ is satisfied for the
$[M;{\mathcal L}]$-formulas $F,G$ and that $\mbox{CF}(J \, F G; \mu; x)$ holds. 
We obtain $\mbox{CF}(F; \mu; x)$ and $\mbox{CF}(G; \mu; x)$. 
Since $F$ and $G$ satisfy Condition $(*)$,
we conclude that $\mbox{CF}({\mathcal C}(F); \mu; x)$ and 
$\mbox{CF}({\mathcal C}(G); \mu; x)$ are both valid.
Therefore $\mbox{CF}(J\,{\mathcal C}(F)\,{\mathcal C}(G); \mu; x)$
and
$\mbox{CF}({\mathcal C}(J\,F G); \mu; x)$ are satisfied.
Since $F$ and $G$ satisfy Condition $(*)$, we obtain
\begin{align}
{\mathcal C}(\mbox{SbF}(J \, F G; \mu; x)) =
{\mathcal C}(J\,F\,\frac{\mu}{x} ~ G\,\frac{\mu}{x}) =
J\,{\mathcal C}(F\,\frac{\mu}{x}) ~\, {\mathcal C}(G\,\frac{\mu}{x})\nonumber \\ = 
J\,{\mathcal C}(F)\,\frac{\mu}{x} \,{\mathcal C}(G)\,\frac{\mu}{x} = 
\mbox{SbF}({\mathcal C}(J \, F G); \mu; x)\,,\nonumber
\end{align}
i.e. Condition $(*)$ is satisfied for $J \, F G$.
\item[(d)]
We assume that Condition $(*)$ is satisfied for a formula $F$ in $[M;{\mathcal L}]$ 
and that there holds $\mbox{CF}(Q\,y\,F; \mu; x)$.
We further keep in mind that $\mbox{free}({\mathcal C}(F)) \subseteq \mbox{free}(F)$
and that ${\mathcal C}(Q\,y\,F)=Q\,y\, {\mathcal C}(F)$.
In the case $x \notin \mbox{free}(F)\setminus \{y\}$ 
we have $x \notin \mbox{free}({\mathcal C}(F))\setminus \{y\}$ and conclude that 
$\mbox{CF}({\mathcal C}(Q\,y\,F);\,\mu;\,x)$ 
as well as $${\mathcal C}(\mbox{SbF}(Q\,y\,F;\,\mu;\,x))=
{\mathcal C}(Q\,y\,F)=
\mbox{SbF}({\mathcal C}(Q\,y\,F);\,\mu;\,x)\,.$$
Otherwise we use that
$\mbox{CF}(Q\,y\,F; \mu; x)$ is satisfied with $x \neq y$ and conclude that
$y \notin \mbox{var}(\mu)$ and $\mbox{CF}(F; \mu; x)$.
Recall that $F$ satisfies the Condition $(*)$ 
which implies $\mbox{CF}({\mathcal C}(F); \mu; x)$. 
From $y \notin \mbox{var}(\mu)$
and $\mbox{CF}({\mathcal C}(F); \mu; x)$ we conclude 
$\mbox{CF}(Q\,y\,{\mathcal C}(F); \mu; x)$, i.e.
$\mbox{CF}({\mathcal C}(Q\,y\,F); \mu; x)$ is again satisfied.
Since $F$ satisfies the Condition $(*)$, we finally conclude due to $x \neq y$ that
\begin{equation*}
\begin{split}
{\mathcal C}(\mbox{SbF}(Q \, y \, F ; \mu ; x))&=
Q \, y \, {\mathcal C}(\mbox{SbF}(F ; \mu ; x))\\
&=Q \, y \, \mbox{SbF}({\mathcal C}(F) ; \mu ; x)\\
&=\mbox{SbF}({\mathcal C}(Q \, y \, F) ; \mu ; x)\,,\\
\end{split}
\end{equation*}
i.e. Condition $(*)$ is satisfied for $Q \, y \, F$.
\end{itemize}
Thus we have proved the lemma.
\end{proof}

\noindent
\textit{Theorem:} With the assumptions of this subsection we obtain that
$$[{\mathcal C}({\Lambda})]=[{\mathcal C}(F_1);\ldots;{\mathcal C}(F_l)]$$ 
is again a proof in $[M;{\mathcal L}]$\,.
\begin{proof} We employ induction with respect to the
rules of inference. First we note that for the ``initial proof"
$[\Lambda]=[\,]$ we can choose $[{\mathcal C}(\Lambda)]=[\,]$.

In the sequel we assume that $[\Lambda]$ 
as well as $[{\mathcal C}(\Lambda)]=[{\mathcal C}(F_1);...;{\mathcal C}(F_l)]$ 
are both proofs in $[M;{\mathcal L}]$.

\begin{itemize}
\item[(a)]
Let $H$ be an axiom in $[M;{\mathcal L}]$.
Then $[\Lambda_{*}] = [\Lambda\,;\,H\,]$ 
is also a proof in $[M;{\mathcal L}]$ 
due to Rule (a). We show that ${\mathcal C}(H)$ is again an axiom. 
Then $[{\mathcal C}(\Lambda_{*})] = [{\mathcal C}(\Lambda)\,;\,{\mathcal C}(H)\,]$ 
is a proof in $[M;{\mathcal L}]$ due to Rule (a).
For this purpose we distinguish four cases.

1.) ~ Let $\alpha=\alpha(\xi_1,...,\xi_m)$ be an identically true
propositional function of the distinct propositional variables $\xi_1,...,\xi_m$, $m \geq 1$. We suppose without loss of generality that all $m$ propositional variables 
occur in $\alpha$.
If $H_1$,...,$H_m$ are formulas in $[M;{\mathcal L}]$ with $H = \alpha(H_1,...,H_m)$,
then 
${\mathcal C}(H) = \alpha({\mathcal C}(H_1),...,{\mathcal C}(H_m))$ 
is again an axiom of the propositional calculus in $[M;{\mathcal L}]$. 

2.) ~If $H$ is an axiom of equality \cite[(3.10)(a),(b)]{Ku}
in $[M;{\mathcal L}]$ then ${\mathcal C}(H)=H$.
For \cite[(3.10)(c)]{Ku}, $p \neq q$ or $n \neq j$
we have again ${\mathcal C}(H)=H$. 
For \cite[(3.10)(c)]{Ku}, $p = q$ and $n = j$ we obtain that
${\mathcal C}(H)$ is an axiom of the propositional calculus.

3.) ~If $H$ is a quantifier axiom \cite[(3.11)]{Ku}
then ${\mathcal C}(H)$ is again a quantifier axiom.
For the quantifier axioms (3.11)(b) we further have to note that 
$x \notin \mbox{free}(F)$ implies $x \notin \mbox{free}({\mathcal C}(F))$.

4.) ~For $H \in B_M$ we obtain again ${\mathcal C}(H)=H$.

\item[(b)]
Let $F$, $G$ be two formulas in $[M;{\mathcal L}]$ and $F$, $\to F\,G$ both steps of the proof 
$[\Lambda]$. Then 
$[\Lambda_{*}] = [\Lambda\,;\,G\,]$ is also a proof in $[M;{\mathcal L}]$ due to Rule (b).
It follows that ${\mathcal C}(F)$ and 
${\mathcal C}(\to F\,G)\,=\,\to {\mathcal C}(F)\,{\mathcal C}(G)$
are both steps of the proof $[{\mathcal C}(\Lambda)]$ due to our assumptions,
and due to Rule (b) we put 
$[{\mathcal C}(\Lambda_{*})] = 
[{\mathcal C}(\Lambda)\,;{\mathcal C}(G)\,]$
for the required proof in $[M;{\mathcal L}]$. 
\item[(c)]
Let $F \in [\Lambda]$, $x \in X$ and $\lambda \in {\mathcal L}$.
Suppose that there holds the condition $\mbox{CF}(F;\lambda;x)$.
Then $[\Lambda_{*}] = [\Lambda\,;\,F\,\frac{\lambda}{x}\,]$ is also a proof 
in $[M;{\mathcal L}]$ due to Rule (c). 
We obtain from the previous lemma that there holds the condition
$\mbox{CF}({\mathcal C}(F);\,\lambda;\,x)$ and the equation 
${\mathcal C}(F\,\frac{\lambda}{x})={\mathcal C}(F)\,\frac{\lambda}{x}$.
Since ${\mathcal C}(F) \in [{\mathcal C}(\Lambda)]$ we conclude that
$[{\mathcal C}(\Lambda_{*})] = 
[{\mathcal C}(\Lambda)\,;{\mathcal C}(F\,\frac{\lambda}{x})\,]$
is a proof in $[M;{\mathcal L}]$ due to Rule (c).
\item[(d)]
Let $F \in [\Lambda]$ and $x \in X$. Then 
$[\Lambda_{*}] = [\Lambda\,;\,\forall\,x\,F\,]$ is also a proof in
$[M;{\mathcal L}]$
due to Rule (d).
Since $F \in [\Lambda]$ implies 
${\mathcal C}(F) \in [{\mathcal C}(\Lambda)]$ and since 
${\mathcal C}(\forall\,x\,F)=\forall\,x\,{\mathcal C}(F)$, we can apply Rule (d) on 
$[{\mathcal C}(\Lambda)]$, ${\mathcal C}(F)$ in order to conclude that
$[{\mathcal C}(\Lambda_{*})] = 
[{\mathcal C}(\Lambda)\,;{\mathcal C}(\forall\,x\,F)\,]$
is a proof in $[M;{\mathcal L}]$.  
\item[(e)] 
In the following we fix a predicate symbol $p \in P_S$,
a list $x_1,...,x_i$ of $i\geq 0$ distinct variables
and a formula $G$ in $[M;{\mathcal L}]$. 
We suppose that $x_1,...,x_i$ and the variables of $G$
are not involved in $B_S$.

Then to every R-formula $F$ of $B_S$ there corresponds exactly one formula $F'$ 
of the mathematical system, which is obtained if we replace in $F$ each 
$i-$ary subformula $p \, \lambda_1,...,\lambda_i$, 
where $\lambda_1,...,\lambda_i$ are lists, by the formula 
$G ~ \frac{\lambda_1}{x_1}...\frac{\lambda_i}{x_i}$.
Note that in this case $\lambda_1,...,\lambda_i \in {\mathcal L}$.
 
Suppose that $F'$ is a step of $[\Lambda]$ for all R-formulas $F \in B_S$ 
for which $p$ occurs $i-$ary in the R-conclusion of $F$.
Then $$[\Lambda_*]=[\Lambda; ~\to ~p\,x_1,...,x_i ~G]$$ 
is also a proof in $[M;{\mathcal L}]$ due to Rule (e).

We distinguish two cases: In the first case we assume that
$p=q$ and $i=j$.
Then we can apply Rule (a) on the formula
$${\mathcal C}(~\to ~p\,x_1,...,x_i ~G)=~\to ~C ~{\mathcal C}(G)\,,$$
which is an axiom of the propositional calculus\,,
and conclude that
\[[{\mathcal C}(\Lambda_{*})] = 
[{\mathcal C}(\Lambda)\,;{\mathcal C}(~\to ~p\,x_1,...,x_i ~G)\,]\]
is a proof in $[M;{\mathcal L}]$. See also \cite[(3.14), Example 2]{Ku}.

In the second case we assume that $p \neq q$ or $i \neq j$.
For every R-formula $F$ of $B_S$ we define the formula $F''$ 
of $[M;{\mathcal L}]$ which is obtained if we replace in $F$ each 
$i-$ary subformula $p \, \lambda_1,...,\lambda_i$  
with $\lambda_1,...,\lambda_i \in {\mathcal L}$ by the formula 
${\mathcal C}(G) ~ \frac{\lambda_1}{x_1}...\frac{\lambda_i}{x_i}$
and note that the variables in ${\mathcal C}(G)$ are not involved
in $B_S$, because we have assumed that the bound variable $z$ in the contradiction $C$ does not occur in $B_S$.
We have assumed that $q$ does not occur $j$-ary in $B_M$, hence in the formula $F'$
the symbol $q$ can only occur $j$-ary within the subformulas 
$\begin{displaystyle} G ~ \frac{\lambda_1}{x_1}...\frac{\lambda_i}{x_i} \end{displaystyle}$.
From the previous lemma we obtain 
\begin{equation*}
{\mathcal C}(G) ~ \frac{\lambda_1}{x_1}...\frac{\lambda_i}{x_i}=
{\mathcal C}(G ~ \frac{\lambda_1}{x_1}...\frac{\lambda_i}{x_i})\,.
\end{equation*}
We see that $F''={\mathcal C}(F' )$ and recall that $[\Lambda]$ as well as 
$[{\mathcal C}(\Lambda)]=[{\mathcal C}(F_1);...;{\mathcal C}(F_l)]$ 
are both proofs in $[M;{\mathcal L}]$.
Hence $F''$ is a step of
the proof $[{\mathcal C}(\Lambda)]$ for all R-formulas $F \in B_S$ 
for which $p$ occurs $i-$ary in the R-conclusion of $F$.
Therefore we can apply Rule (e) on $[{\mathcal C}(\Lambda)]$ and conclude that
\[[{\mathcal C}(\Lambda_{*})] = 
[{\mathcal C}(\Lambda)\,;{\mathcal C}(~\to ~p\,x_1,...,x_i ~G)\,]\]
with ${\mathcal C}(~\to ~p\,x_1,...,x_i ~G)
=~\to ~p\,x_1,...,x_i ~{\mathcal C}(G)$ is a proof in $[M;{\mathcal L}]$. 
\end{itemize}
Thus we have proved the theorem.
\end{proof} 

The lemma and theorem of this subsection have a strong resemblance
to \cite[(4.7) Lemma, (4.8) Theorem]{Ku}, and the proofs are very similar.
It arises the question whether there is a more general result which
is relevant in elementary proof theory.

\subsection{A general theorem concerning formal induction and its application to PA}\label{pa}\hfill\\
\noindent
We consider a mathematical system $M=[S;A_M;P_M;B_M]$ with an underlying recursive system $S=[A_S;P_S;B_S]$ such that $A_M = A_S$, $P_M = P_S$ and 
$B_M = B_S$ and assume that $[M;{\mathcal L}]$ is a mathematical system with restricted argument lists in ${\mathcal L}$\,.
We suppose that ${\mathcal L}$ is enumerable, for more details see the text introducing \cite[(5.4) Conjecture]{Ku}.
We will study this mathematical system
$[M;{\mathcal L}]$ until we discuss its application to Peano arithmetic PA.

\begin{defin}\label{sl}
An R-derivation $[\Lambda]$ in $[S;{\mathcal L}]$ is defined as an R-derivation in $S$
with  the following restrictions: The R-formulas in $[\Lambda]$ and 
the R-formulas $F$, $G$ in \cite[(1.11)]{Ku} have only argument lists in 
${\mathcal L}$,
and the use of the Substitution Rule \cite[(1.11)(c)]{Ku} is restricted 
to $\lambda \in {\mathcal L}$.
Then the R-formulas in $[\Lambda]$ are called R-derivable 
in $[S;{\mathcal L}]$. By
$\Pi_R(S;{\mathcal L})$ we denote the set of all 
R-derivable R-formulas in $[S;{\mathcal L}]$\,,
by $\Pi(M;{\mathcal L})$ the set of all 
provable formulas in $[M;{\mathcal L}]$\,.
\end{defin}

\noindent
From \cite[Section 3, Example 2]{Ku} we know
that the formula $\neg q \, x_1,\ldots,x_j$ with $x_1,\ldots,x_j \in X$
is provable in $[M;{\mathcal L}]$
whenever $q \in P_S$ does not occur $j$-ary in $B_S$.
On the other hand we have shown the consistency of $[M;{\mathcal L}]$
in \cite[(5.1) Proposition]{Ku}.

\noindent
We will first simplify the syntax of the formulas $F$
in $[M;{\mathcal L}]$ by re\-moving the quantifier $\exists$
and the symbols $\vee$, $\&$ and $\leftrightarrow$.
By $\mbox{Form}(M;{\mathcal L})$ we denote the 
set of all formulas in $[M;{\mathcal L}]$.
Let ${\mathcal F}$ be the set of all formulas in 
$\mbox{Form}(M;{\mathcal L})$ without the symbols 
$\exists$, $\vee$, $\&$ and $\leftrightarrow$
and define the mapping $\Theta : {\mbox Form}(M;{\mathcal L}) \mapsto
{\mathcal F}$ as follows:
\begin{itemize}
\item[1.] $\Theta(F)=F$ if $F$ is a prime-formula in $[M;{\mathcal L}]$.
\item[2.] $\Theta(\neg F)=\neg \Theta(F)$ for all
formulas $F$ in $[M;{\mathcal L}]$.
\item[3.] For all $F, G \in \mbox{Form}(M;{\mathcal L})$ we have
\begin{itemize}
\item[i.] $\Theta(\rightarrow F G) \, = \, \rightarrow \Theta(F) \Theta(G)$.
\item[ii.] $\Theta(\vee F G) \, = \,\rightarrow \neg \Theta(F) \Theta(G)$.
\item[iii.] $\Theta(\& F G) \, = \,\neg \rightarrow \Theta(F) \neg \Theta(G)$.
\item[iv.] $\Theta(\leftrightarrow F G) \, = \,\neg \,\rightarrow \,\rightarrow \Theta(F) \Theta(G) \neg \rightarrow \Theta(G) \Theta(F)$.
\end{itemize}
\item[4.]
\begin{itemize}
\item[i.] $\Theta(\forall x F)=\forall x \Theta(F)$ for all $x \in X$ and
$F \in \mbox{Form}(M;{\mathcal L})$.
\item[ii.] $\Theta(\exists x F)=
\neg \forall x \neg \Theta(F)$ for all $x \in X$, 
$F \in \mbox{Form}(M;{\mathcal L})$.
\end{itemize}
\end{itemize}
\begin{thm}\label{simply}
Let $[\Lambda]=[F_1;\ldots;F_l]$ be a proof in $[M;{\mathcal L}]$ with the steps $F_1 ; \ldots; F_l$. Then
$$[\Theta({\Lambda})]=[\Theta(F_1);\ldots;\Theta(F_l)]$$ 
is again a proof in $[M;{\mathcal L}]$\,.
For all $k=1,\ldots,l$ the formula $\Theta(F_k)$ can be derived with
the same rule of inference that was used for the derivation of $F_k$ 
in the proof $[\Lambda]$.
\end{thm}
\begin{proof}
We employ induction with respect to the
rules of inference. First we note that for the ``initial proof"
$[\Lambda]=[\,]$ we can choose $[\Theta(\Lambda)]=[\,]$.

In the sequel we assume that $[\Lambda]$ 
as well as $[\Theta(\Lambda)]=[\Theta(F_1);...;\Theta(F_l)]$ 
are both proofs in $[M;{\mathcal L}]$.

\begin{itemize}
\item[(a)]The basis axioms in $[M;{\mathcal L}]$ are exactly the basis R-axioms of the underlying recursive system, hence $\Theta(F)=F$ for all 
formulas $F \in B_M=B_S$.
If $F$ is an axiom of equality then we have again $\Theta(F)=F$.
If $F$ is an axiom of the propositional calculus, then also $\Theta(F)$.
If $F$ is an quantifier axiom (3.11)(a),(b), then also $\Theta(F)$.
If $F$ is an quantifier axiom (3.11)(c), then $\Theta(F)$
is an axiom of the propositional calculus.
We conclude that $\Theta$ maps axioms into axioms.
\item[(b)]
Let $F$, $G$ be two formulas in $[M;{\mathcal L}]$ and $F$, $\to F\,G$ both steps of the proof $[\Lambda]$. Then 
$[\Lambda_{*}] = [\Lambda\,;\,G\,]$ is also a proof in $[M;{\mathcal L}]$ due to Rule (b).
It follows that 
$$\Theta(F) \quad \mbox{and} \quad
\Theta(\to F\,G)\,=\,\to \Theta(F)\,\Theta(G)
$$
are both steps of the proof $[\Theta(\Lambda)]$ due to our assumptions,
and due to Rule (b) we put 
$[\Theta(\Lambda_{*})] = 
[\Theta(\Lambda)\,;\Theta(G)\,]$
for the required proof in $[M;{\mathcal L}]$. 
\item[(c)]
Let $F \in [\Lambda]$, $x \in X$ and $\lambda \in {\mathcal L}$.
Suppose that there holds the condition $\mbox{CF}(F;\lambda;x)$.
Then $[\Lambda_{*}] = [\Lambda\,;\,F\,\frac{\lambda}{x}\,]$ is also a proof 
in $[M;{\mathcal L}]$ due to Rule (c). 
We obtain that there holds the condition
$\mbox{CF}(\Theta(F);\,\lambda;\,x)$ and the equation 
$\Theta(F\,\frac{\lambda}{x})=\Theta(F)\,\frac{\lambda}{x}$.
Since $\Theta(F) \in [\Theta(\Lambda)]$ we conclude that
$[\Theta(\Lambda_{*})] = 
[\Theta(\Lambda)\,;\Theta(F\,\frac{\lambda}{x})\,]$
is a proof in $[M;{\mathcal L}]$ due to Rule (c).
\item[(d)]
Let $F \in [\Lambda]$ and $x \in X$. Then 
$[\Lambda_{*}] = [\Lambda\,;\,\forall\,x\,F\,]$ is also a proof in
$[M;{\mathcal L}]$
due to Rule (d).
Since $F \in [\Lambda]$ implies 
$\Theta(F) \in [\Theta(\Lambda)]$ and since 
$\Theta(\forall\,x\,F)=\forall\,x\,\Theta(F)$, we can apply Rule (d) on 
$[\Theta(\Lambda)]$, $\Theta(F)$ in order to conclude that
$[\Theta(\Lambda_{*})] = 
[\Theta(\Lambda)\,;\Theta(\forall\,x\,F)\,]$
is a proof in $[M;{\mathcal L}]$.  
\item[(e)] 
In the following we fix a predicate symbol $p \in P_S$,
a list $x_1,...,x_i$ of $i\geq 0$ distinct variables
and a formula $G$ in $[M;{\mathcal L}]$. 
We suppose that $x_1,...,x_i$ and the variables of $G$
are not involved in $B_S$.

Then to every R-formula $F$ of $B_S$ there corresponds exactly one formula $F'$ 
of the mathematical system, which is obtained if we replace in $F$ each 
$i-$ary subformula $p \, \lambda_1,...,\lambda_i$, 
where $\lambda_1,...,\lambda_i$ are lists, by the formula 
$G ~ \frac{\lambda_1}{x_1}...\frac{\lambda_i}{x_i}$.
Note that in this case $\lambda_1,...,\lambda_i \in {\mathcal L}$.
 
Suppose that $F'$ is a step of $[\Lambda]$ for all R-formulas $F \in B_S$ 
for which $p$ occurs $i-$ary in the R-conclusion of $F$.
Then $$[\Lambda_*]=[\Lambda; ~\to ~p\,x_1,...,x_i ~G]$$ 
is also a proof in $[M;{\mathcal L}]$ due to Rule (e).

For every R-formula $F$ of $B_S$ we define the formula $F''$ 
of $[M;{\mathcal L}]$ which is obtained if we replace in $F$ each 
$i-$ary subformula $p \, \lambda_1,...,\lambda_i$  
with $\lambda_1,...,\lambda_i \in {\mathcal L}$ by the formula 
$\Theta(G) ~ \frac{\lambda_1}{x_1}...\frac{\lambda_i}{x_i}$
and note that the variables in $\Theta(G)$ are not involved
in $B_S$, because $\Theta(G)$ and $G$ both have the same variables.
We see that $F''=\Theta(F' )$ and recall that $[\Lambda]$ and
$[\Theta(\Lambda)]=[\Theta(F_1);...;\Theta(F_l)]$ 
are both proofs in $[M;{\mathcal L}]$.
Hence $F''$ is a step of
the proof $[\Theta(\Lambda)]$ for all R-formulas $F \in B_S$ 
for which $p$ occurs $i-$ary in the R-conclusion of $F$.
Therefore we can apply Rule (e) on $[\Theta(\Lambda)]$ and conclude that
\[[\Theta(\Lambda_{*})] = 
[\Theta(\Lambda)\,;\Theta(~\to ~p\,x_1,...,x_i ~G)\,]\]
with $\Theta(~\to ~p\,x_1,...,x_i ~G)
=~\to ~p\,x_1,...,x_i ~\Theta(G)$ is a proof in $[M;{\mathcal L}]$. 
\end{itemize}
Thus we have proved the theorem.
\end{proof}

\noindent
Now we will roughly divide the formulas in ${\mathcal F}$ into equivalence classes. This is used in order to present a well defined interpretation of the formulas $F \in {\mathcal F}$ in the mathematical system.
\noindent
\begin{defin}\label{classdef} Equivalence classes $\langle F \rangle$ of formulas 
$F \in {\mathcal F}$.
\begin{itemize}
\item[1.] By ${\mathcal P}$ we denote the set of all prime formulas in $[M;{\mathcal L}]$.
For any prime formula $F \in {\mathcal P}$ we have $F \in {\mathcal F}$ and put 
$\langle F \rangle = {\mathcal P}$ .
\item[2.] For $F \in {\mathcal F}$ we also have $\neg F \in {\mathcal F}$ and put
\begin{equation*}
\langle \neg F \rangle = \neg \langle F \rangle =
\left\{\neg F' \,:\,F' \in \langle F \rangle\, \right\}\,.
\end{equation*}
\item[3.] For $F, G \in {\mathcal F}$ we also have $\rightarrow\,F G \in {\mathcal F}$ and put
\begin{equation*}
\langle \rightarrow F G \rangle = \,\rightarrow \langle F \rangle \langle G \rangle =
\left\{\rightarrow F' G'  : F' \in \langle F \rangle, \, G' \in \langle G \rangle\, \right\}\,.
\end{equation*}
\item[4.] For $x \in X$, $F \in {\mathcal F}$ 
we also have $\forall x F \in {\mathcal F}$ and put
\begin{equation*}
\langle \forall\, x\, F \rangle =\forall \langle F \rangle =
\left\{\forall x' F' \,:\,x' \in X, \, F' \in \langle F \rangle\, \right\}\,.
\end{equation*}
\end{itemize}
\end{defin}
\noindent
The sets $\langle F \rangle$ with $F \in {\mathcal F}$ give a well-defined partition of 
${\mathcal F}$, two formulas $F$ and $F'$ in ${\mathcal F}$ are equivalent if and only if 
$\langle F \rangle =\langle F' \rangle$
for their equivalence classes $\langle F \rangle$ and $\langle F' \rangle$, respectively.
The construction of each class starts with ${\mathcal P}$ and terminates in a finite number of steps. It is purely syntactic, for example
$\neg \neg \forall \rightarrow {\mathcal P}{\mathcal P}$ and 
$\forall \rightarrow {\mathcal P}{\mathcal P}$ are disjoint.

\noindent
By ${\mathcal F}_*$ we denote the set of all formulas in ${\mathcal F}$
without free variables, also called \textit{statements}.
Let ${\mathcal L}_*$ be the set of all lists in ${\mathcal L}$ without variables. We suppose that ${\mathcal L}_*$ is not empty.
Now we will give an interpretation of all statements $F \in {\mathcal F}_*$ 
in the mathe\-matical system $[M;{\mathcal L}]$. 

\noindent
Using the verum $\top$, the empty set $\emptyset$ and formulas $F, G \in {\mathcal F}$
we define the following function 
$\mbox{V} \, : \, {\mathcal F}_* \, \mapsto \{\emptyset,\{\top\}\}$:
\begin{itemize}
\item[1.] If $\lambda, \mu \in {\mathcal L}_*$ then
\begin{equation*}
\mbox{V}(\sim \lambda, \mu)=
\begin{cases}
~\{\top\} \, &\ \text{if} \quad  \sim \lambda, \mu \in \Pi_R(S;{\mathcal L}) ,\\
~~ \emptyset \, &\ \text{otherwise}\,.\\
\end{cases}
\end{equation*}
Let $p \in P_S$
and $\lambda_1,...,\lambda_i \in {\mathcal L}_*$ for $i \geq 0$ be elementary $A_S$-lists in ${\mathcal L}_*$. Then we evaluate
\begin{equation*}
\mbox{V}(p\,\lambda_1,...,\lambda_i)=
\begin{cases}
~\{\top\} \, &\ \text{if} \quad  p\,\lambda_1,...,\lambda_i \in \Pi_R(S;{\mathcal L}) ,\\
~~ \emptyset \, &\ \text{otherwise}\,.\\
\end{cases}
\end{equation*}
\item[2.] For $\neg \,F \in {\mathcal F}_*$ we also have $F \in {\mathcal F}_*$ and require
\begin{equation*}
\mbox{V}(\neg\,F)=\{\top\} \,\setminus \mbox{V}(F)\,.
\end{equation*}
\item[3.] For $\rightarrow\,F G \in {\mathcal F}_*$ we also have 
$F, G \in {\mathcal F}_*$ and require
\begin{equation*}
\mbox{V}(\rightarrow\,F G) =\left(\{\top\} \,\setminus \mbox{V}(F)\right)\,
\cup \, \mbox{V}(G)\,.
\end{equation*}
\item[4.] For $x \in X$, $\forall\,x\,F \in {\mathcal F}_*$ and $\lambda \in {\mathcal L}_*$ 
we have $F\frac{\lambda}{x} \in {\mathcal F}_*$,
recall ${\mathcal L}_* \neq \emptyset$ and require
\begin{equation*}
\mbox{V}(\forall\,x\, F) =\bigcap \limits_{\lambda \in {\mathcal L}_*} 
\mbox{V}\left(F\frac{\lambda}{x}\right)\,.
\end{equation*}
\end{itemize}
\noindent
\noindent
We say that a formula $F \in {\mathcal F}_*$ is true if and only if 
$\top \in \mbox{V}(F)$. 
The sets $\langle F \rangle_{*}=\langle F \rangle \cap {\mathcal F}_*$ 
with $F \in {\mathcal F}_*$ form a partition of ${\mathcal F}_*$, 
and induction on the equivalence classes $\langle F \rangle_{*}$ in ${\mathcal F}_*$
shows that the function $\mbox{V}$ is well-defined. Of course, in general the
evaluation of $V(F)$ must be highly non-constructive.

\begin{defin}\label{gen}
Let $F$ be a formula in $[M;{\mathcal L}]$.
Let $x_1={\bf x_{j_1}}$,$\ldots$, $x_m={\bf x_{j_m}}$ be the list of all free variables in $F$, ordered with increasing indizes
${\bf j_1}<\ldots<{\bf j_m}$ of the variables.
We define 
$$\mbox{Free}(F)=[x_1;\ldots;x_m]\,, \quad
\mbox{Gen}(F)=\forall x_1 \ldots \forall x_m F\,,$$ 
namely the list $\mbox{Free}(F)$ of free variables in $F$ 
and the generalization of the formula $F$, respectively.
Especially for statements $F$ we have $\mbox{Free}(F)=[\,]$
and $\mbox{Gen}(F)=F$.
\end{defin}

Now we make use of Theorem \ref{simply} which guarantees
that proofs with formulas in ${\mathcal F}$ are 
not a real restriction and present our main result, namely

\begin{thm}\label{mainthm}
Let $F \in {\mathcal F}$ be a formula which is provable
in $[M;{\mathcal L}]$. Suppose that the set ${\mathcal L}_*$ 
of all lists in ${\mathcal L}$ without variables is not empty
and that ${\mathcal L}$ is enumerable. Then $\top \in V(\mbox{Gen}(F))$\,.
\end{thm}
\begin{proof}
Let $F \in {\mathcal F}$ be a formula in $[M;{\mathcal L}]$ 
with $\mbox{Free}(F)=[x_1;\ldots;x_m]$. Then $\top \in V(\mbox{Gen}(F))$
iff $\top \in V(F\frac{\lambda_1}{x_1}\ldots\frac{\lambda_m}{x_m})$
for all $\lambda_1, \ldots \lambda_m \in {\mathcal L}_*$.
This will be used throughout the whole proof.

We want to show for each proof $[\Lambda]$ in $[M;{\mathcal L}]$ 
with steps only in ${\mathcal F}$ that $\top \in V(\mbox{Gen}(F))$
for all $F \in [\Lambda]$. We employ induction with respect to the
rules of inference. 
First we note that the statement is true for the ``initial proof"
$[\Lambda]=[\,]$.

\noindent
Let $[\Lambda]=[F_1;\ldots;F_l]$ be a proof in $[M;{\mathcal L}]$ with the steps $F_1 ; \ldots; F_l \in {\mathcal F}$. 
Our induction hypothesis is
$\top \in V(\mbox{Gen}(F_k))$
for all $k=1,\ldots,l$\,.
\begin{itemize}
\item[(a)] Here we show that $\top \in V(\mbox{Gen}(F))$
for all axioms $F \in {\mathcal F}$. Then
the extended proof $[\Lambda_*]=[\Lambda;F]$
will also satisfy the statement.
\begin{itemize}
\item[$\bullet$] The basis axioms and the axioms of equality in 
$[M;{\mathcal L}]$ are R-axioms of the underlying recursive system. Assume that $F$ is such an axiom with $\mbox{Free}(F)=[x_1;\ldots;x_m]$ and that 
$\lambda_1, \ldots \lambda_m \in {\mathcal L}_*$. 
Then $\tilde{F}=F\frac{\lambda_1}{x_1}\ldots\frac{\lambda_m}{x_m}$
is an elementary R-formula in ${\mathcal F}$. We have
$\top \in V(\tilde{F})$ iff there is an R-premise in 
$\tilde{F}$ which is not R-derivable in $[S;{\mathcal L}]$
or if the R-conclusion of $\tilde{F}$ is R-derivable in $[S;{\mathcal L}]$.
But due to the Modus Ponens Rule the R-conclusion of $\tilde{F}$ is 
R-derivable in $[S;{\mathcal L}]$ if all R-premises in $\tilde{F}$ 
are R-derivable in $[S;{\mathcal L}]$. 
We see $\top \in V(\tilde{F})$ and hence
$\top \in V(\mbox{Gen}(F))$ in the case that $F$ is a basis axiom or an axiom of equality in the mathematical system $[M;{\mathcal L}]$.
\item[$\bullet$] Suppose that $\alpha=\alpha(\xi_1,...,\xi_j)$ is an identically true propositional function defined in \cite[(3.8)]{Ku} which is
only constructed with the negation symbol ``$\neg$'' and the
implication arrow ``$\rightarrow$'' and that 
$F_1$,...,$F_j \in {\mathcal F}$ are formulas in $[M;{\mathcal L}]$. 
Then the formula $F = \alpha(F_1,...,F_j) \in {\mathcal F}$ is an axiom of the propositional calculus.
Prescribe $\lambda_1, \ldots \lambda_m \in {\mathcal L}_*$
and put
$$
\tilde{F}=F\frac{\lambda_1}{x_1}\ldots\frac{\lambda_m}{x_m}\,, ~
\tilde{F_k}=F_k\frac{\lambda_1}{x_1}\ldots\frac{\lambda_m}{x_m}
$$
for $k=1,\ldots,j$ and $\mbox{Free}(F ) = [x_1;\ldots;x_m]$\,. 
For any two formulas $F', F'' \in {\mathcal F}_*$
we have $\top \in V(\neg F')$ iff $\top \notin V(F')$ 
and $\top \in V(\rightarrow F' F'')$ iff
$\top \in V(F')$ implies $\top \in V(F'')$, respectively.
We see that 
$
\tilde{F}=\alpha(\tilde{F_1}, \ldots, \tilde{F_j}) \in \mathcal{F}_*
$
is an axiom of the propositional calculus with $\top \in V(\tilde{F})$.
\item[$\bullet$] Suppose that $x \in X$, that $F \in \mathcal{F}$
and $$\mbox{Free}(\forall x F ) = [x_1;\ldots;x_m]\,.$$
We put $H = \, \rightarrow \forall x F ~ F$ and 
note that $x \notin [x_1;\ldots;x_m]$. We see
$\top \in V(\mbox{Gen}(H))$ iff
$$\top \in V(\mbox{SbF}(\tilde{H};\mu;x)) = 
V(\rightarrow \forall x \tilde{F} ~ 
\mbox{SbF}(\tilde{F};\mu;x))$$
for all $\mu, \lambda_1, \ldots \lambda_m \in {\mathcal L}_*$,
using $\tilde{F} = F\frac{\lambda_1}{x_1}\ldots\frac{\lambda_m}{x_m}$ and
$$
\tilde{H} = H\frac{\lambda_1}{x_1}\ldots\frac{\lambda_m}{x_m}
= \, \rightarrow \forall x \tilde{F} ~ \tilde{F}
$$
as abbreviations. 

Now $\top \in V(\forall x \tilde{F})$
implies indeed $\top \in V(\mbox{SbF}(\tilde{F};\mu;x))$
for all $\mu, \lambda_1, \ldots \lambda_m \in {\mathcal L}_*$,
independent of $x \in \mbox{Free}(F)$ or $x \notin \mbox{Free}(F)$.
\item[$\bullet$] Suppose that $x \in X$, that $F, G \in \mathcal{F}$
and that $x \notin \mbox{Free}(F)$, 
$\mbox{Free}(\forall x \rightarrow F G)=
\mbox{Free}(\rightarrow F \, \forall x G)=[x_1;\ldots;x_m]$. We put 
$$H = \, \rightarrow \forall x \rightarrow F G ~ \rightarrow F \, \forall x G\,,$$ fix arbitrary lists
$\lambda_1, \ldots \lambda_m \in {\mathcal L}_*$ and make use of the abbreviations
$\tilde{F} = F\frac{\lambda_1}{x_1}\ldots\frac{\lambda_m}{x_m}$ and
$\tilde{G} = G\frac{\lambda_1}{x_1}\ldots\frac{\lambda_m}{x_m}$.
We have
$$
\tilde{H} = H\frac{\lambda_1}{x_1}\ldots\frac{\lambda_m}{x_m}
= \, \, \rightarrow \forall x \rightarrow \tilde{F} \tilde{G} ~ 
\rightarrow \tilde{F} \, \forall x \tilde{G}
$$
with $\tilde{H} \in \mathcal{F}_*$\,.
In order to show $\top \in V(\tilde{H})$ we assume 
$\top \in V( \forall x \rightarrow \tilde{F} \tilde{G})$
and note that $x \notin \mbox{Free}(\tilde{F})$.
Then
$$\qquad \top \in V( \forall x \rightarrow \tilde{F} \tilde{G})
~\mbox{iff}~
\top \in V( \rightarrow \tilde{F} \, \mbox{SbF}(\tilde{G};\lambda;x))$$
for all $\lambda \in {\mathcal L}_*$.
Hence $\top \in V(\tilde{F})$
implies $\top \in V(\mbox{SbF}(\tilde{G};\lambda;x))$ for all 
$\lambda \in \mathcal{L}_*$, i.e. $\top \in V(\tilde{F})$
implies $\top \in V(\, \forall x \tilde{G})$,
and we have shown $\top \in V(\,\to \tilde{F}\, \forall x \tilde{G})$
and $\top \in V(\tilde{H})$.
\item[$\bullet$] Recall that the quantifier axiom (3.11)(c) is 
replaced by an axiom of the propositional calculus
due to Theorem \ref{simply}.
\end{itemize}
\item[(b)] Suppose that $F$ and $H=\rightarrow F G$ are both steps
of the proof $[\Lambda]=[F_1;\ldots;F_l]$
with $\mbox{Free}(\rightarrow F G)=[x_1;\ldots;x_m]$\,.
Then $\top \in V(\mbox{Gen}(F))$ and 
$\top \in V(\mbox{Gen}(H))$ from our induction hypothesis. Fix
$\lambda_1, \ldots \lambda_m \in {\mathcal L}_*$ and put
$\tilde{F} = F\frac{\lambda_1}{x_1}\ldots\frac{\lambda_m}{x_m}$,
$\tilde{G} = G\frac{\lambda_1}{x_1}\ldots\frac{\lambda_m}{x_m}$.\\

For 
$$
\tilde{H} = H\frac{\lambda_1}{x_1}\ldots\frac{\lambda_m}{x_m}
= \, \, \rightarrow \tilde{F} \tilde{G} 
$$
we have $\tilde{F}, \tilde{G}, \tilde{H} \in 
\mathcal{F}_*$, $\top \in V(\tilde{F})$,
$\top \in V(\tilde{H})$ and $\top \in V(\tilde{G})$.
Note that substitutions of variables in $[x_1;\ldots;x_m]$
not occurring in $F$ or $G$ are allowed, because they do not have any effect.
We obtain that the extended proof $[\Lambda_*]=[\Lambda;G]$
also satisfies our statement.
\item[(c)] Let $x \in X$ and suppose that $F \in \mathcal{F}$
is a step of the proof $[\Lambda]=[F_1;\ldots;F_l]$. 
Let $\lambda \in {\mathcal L}$ and
suppose that there holds the condition $\mbox{CF}(F;\lambda;x)$. 
Note that $\top \in V(\mbox{Gen}(F))$ from our induction hypothesis.

Without loss of generality we may assume that $x \in \mbox{free}(F)$,
where we use the \textit{set} $\mbox{free}(F)=\{x,x_1,\ldots,x_m\}$
(instead of ordered lists) with distinct variables $x,x_1,\ldots,x_m \in X$,
and put
$$
\Phi(F) = \{F\frac{\lambda_0}{x}
\frac{\lambda_1}{x_1}\ldots\frac{\lambda_m}{x_m}\,:
\,\lambda_0,\lambda_1,\ldots,\lambda_m \in \mathcal{L}_*\,\}\,.
$$
We write $\mbox{var}(\lambda) =\{y_1,\ldots,y_k\}$ and 
$\lambda = \lambda(y_1,\ldots,y_k)$\,.
From $x \in \mbox{free}(F)$ and $\mbox{CF}(F;\lambda;x)$ we see that 
$$\mbox{var}(\lambda)  \subseteq \mbox{free}\left(F\frac{\lambda}{x}\right)\,.$$
Hence we can write
$\mbox{free}(F\frac{\lambda}{x})=\{y_1,\ldots,y_n\}$
with $n \geq k$ distinct variables $y_1,\ldots,y_n\in X$
and define the new set
$$
\Phi(F;\lambda;x)=\{F\frac{\lambda}{x}
\frac{\mu_1}{y_1}\ldots\frac{\mu_n}{y_n}\,:
\,\mu_1,\ldots,\mu_n \in \mathcal{L}_*\,\}\,.
$$
Again from $\mbox{CF}(F;\lambda;x)$ we conclude that
$$\qquad
\Phi(F;\lambda;x)=\left\{F\frac{\lambda(\mu_1,\ldots,\mu_k)}{x}
\frac{\mu_1}{y_1}\ldots\frac{\mu_n}{y_n}\,:
\,\mu_1,\ldots,\mu_n \in \mathcal{L}_*\,\right\}\,,
$$
hence $\Phi(F;\lambda;x) \subseteq \Phi(F)$ and
\begin{equation*}
\begin{split}
   V(\mbox{Gen}(F))&=\bigcap \limits_{G \in \Phi(F)} V(G)\\
		&\subseteq \bigcap \limits_{G \in \Phi(F;\lambda;x)} V(G)
      = V\left(\mbox{Gen}\left(F\frac{\lambda}{x}\right)\right)\,.\\
\end{split}
\end{equation*}
We obtain from our induction hypothesis $\top \in V(\mbox{Gen}(F))$
that $\top \in V(\mbox{Gen}(F\frac{\lambda}{x}))$.
Now the extended proof $[\Lambda_*]=[\Lambda;F\frac{\lambda}{x}]$
satisfies our statement.
\item[(d)] Let $F$ be a step of the proof $[\Lambda]=[F_1;\ldots;F_l]$. 
If $x \in X$ is not a free variable of $F$ then
$F\frac{\lambda}{x}=F$ for all $\lambda \in {\mathcal L}_*$ and
$$V(\mbox{Gen}(\forall x \,F))=V(\mbox{Gen}(F))\,.$$
Then $\top \in V(\mbox{Gen}(\forall x \,F))$ from our induction hypothesis.
Now we suppose that $x \in \mbox{free}(F)=\{x_1,\ldots,x_m\}$ with distinct variables $x_1,\ldots,x_m$. In this case we see
$\top \in V(\mbox{Gen}(\forall x \,F))$ 
iff for all $\lambda_1,\ldots,\lambda_m \in \mathcal{L}_*$ 
$$ \top \in V\left(
F\frac{\lambda_1}{x_1}\ldots\frac{\lambda_m}{x_m}
\right)\,,$$
i.e. $V(\mbox{Gen}(\forall x \,F))=V(\mbox{Gen}(F))$, and obtain
$\top \in V(\mbox{Gen}(\forall x \,F))$ 
again from our induction hypothesis.
In any case the extended proof $[\Lambda_*]=[\Lambda;\forall x\,F]$
satisfies our statement.
\item[(e)] In the following we fix a predicate symbol $p \in P_S$,
a list $x_1,...,x_i$ of $i\geq 0$ distinct variables
and a formula $G \in {\mathcal F}$. 
We suppose that $x_1,...,x_i$ and the variables of $G$
are not involved in $B_S$.
Then to every R-formula $F$ of $B_S$ there corresponds exactly one formula 
$F'\in {\mathcal F}$ 
of the mathematical system, which is obtained if we replace in $F$ each 
$i-$ary subformula $p \, \lambda_1,...,\lambda_i$, 
where $\lambda_1,...,\lambda_i$ are lists, by the formula 
$G ~ \frac{\lambda_1}{x_1}...\frac{\lambda_i}{x_i}$.
We suppose that $F'$ is a step of a proof $[\Lambda]$ 
for all R-formulas $F \in B_S$ 
for which $p$ occurs $i-$ary in the R-conclusion of $F$.
\noindent
Now $[\Lambda_*]=[\Lambda; ~ \to ~p\,x_1,...,x_i ~G]$ is also a proof
in $[M;{\mathcal L}]$ with formulas in ${\mathcal F}$. 
To finish the proof of the main theorem it remains to show that
$\top \in V(\mbox{Gen}(~ \to ~p\,x_1,...,x_i ~G))$\,.
We may write 
$\mbox{free}(~ \to ~p\,x_1,...,x_i ~G)=\{x_1,...,x_m\}$
with $m \geq i$ distinct variables $x_1,...,x_m$. For $m>i$ we choose 
$\tilde{\lambda}_{i+1},\ldots,\tilde{\lambda}_{m} \in {\mathcal L}_*$
arbitrary but fixed and put
$$\tilde{G} = G\frac{\tilde{\lambda}_{i+1}}{x_{i+1}}
\ldots\frac{\tilde{\lambda}_{m}}{x_{m}}\,,$$
and otherwise we put $\tilde{G}=G$.
It is sufficient to show that
$\top \in V(\mbox{Gen}(~ \to ~p\,x_1,...,x_i ~\tilde{G}))$
with the formula $\tilde{G} \in \mathcal{F}$ satisfying
$\mbox{free}(\tilde{G}) \subseteq\{x_1,...,x_i\}\,.$
Note that the variables of $\tilde{G}$ are not involved in $B_S$.
For $\lambda_{1},\ldots,\lambda_{i} \in {\mathcal L}$ we 
can also write
$$\tilde{G}(\lambda_{1},\ldots,\lambda_{i}) 
= \tilde{G}\frac{\lambda_{1}}{x_{1}}\
\ldots\frac{\lambda_{i}}{x_{i}}\,,$$
provided that $x_1,...,x_i$ and the variables of $G$
are not involved in $\lambda_{1},\ldots,\lambda_{i}$.
Especially for $i=0$ we put 
$\tilde{G}(\lambda_{1},\ldots,\lambda_{i})  = \tilde{G}$.

\noindent
We have to show that 
\begin{equation*}
p \lambda_1,\ldots,\lambda_i \in \Pi_R(S;{\mathcal L})
\Rightarrow
\top \in V(\tilde{G}(\lambda_{1},\ldots,\lambda_{i}))
\end{equation*}
for all $\lambda_{1},\ldots,\lambda_{i} \in {\mathcal L}_*$, 
see Definition \ref{sl}.

\noindent
We will show that $\tilde{G}(\lambda_{1},\ldots,\lambda_{i})$
can be derived in $[M;{\mathcal L}]$ 
from the given proof $[\Lambda]=[F_1;\ldots;F_l]$ 
by using only Rules (a)-(d)
whenever $p \lambda_1,\ldots,\lambda_i \in \Pi_R(S;{\mathcal L})$
for $\lambda_{1},\ldots,\lambda_{i} \in {\mathcal L}_*$.
Then we can first apply Theorem \ref{simply} 
in order to obtain an extension of the proof $[\Lambda]$
which consists only on formulas in ${\mathcal F}$
and which contains the formula $\tilde{G}(\lambda_{1},\ldots,\lambda_{i})$
as a final step.
This will conclude the proof of the theorem because $[\Lambda]$ 
satisfies the induction hypothesis and Rules (a)-(d)
applied step by step on the extensions of $[\Lambda]$ 
with formulas in ${\mathcal F}$
can only produce further new formulas $F$ satisfying 
$\top \in V(\mbox{Gen}(F))$.\\

\noindent
Following our strategy we can construct an algorithm $\mathcal{A}$
with the following properties:
\begin{itemize}
\item[$\bullet$] $\mathcal{A}$ generates an infinite sequence
$R_1; R_2; R_3; \ldots$ of R-for\-mu\-las 
such that each finite part $[R_1; \ldots; R_n]$ with $n \in \setN$
is an R-derivation in $[S;{\mathcal L}]$. Note that $\mathcal{A}$
makes only use of the rules of inference (1.11)(a),(b),(c) in \cite{Ku}.
\item[$\bullet$] All elementary prime R-formulas in $\Pi_R(S;{\mathcal L})$
occur at least one time in the sequence $R_1; R_2; R_3; \ldots$.
\item[$\bullet$]  We suppose that $x_1,...,x_i$ and the variables of $G$
are not involved in $R_1; R_2; R_3;\ldots$, which is not a real restriction.
\end{itemize}

\noindent
Depending on $\mathcal{A}$ we define a second algorithm $\mathcal{B}$
with the following properties:
\begin{itemize}
\item[$\bullet$]  
Algorithm $\mathcal{B}$ generates a (finite or infinite) sequence
of formulas $F_1; F_2; F_3; \ldots$ in $[M;{\mathcal L}]$.
Each finite part $[F_1; \ldots; F_n]$ of the sequence
is a proof in $[M;{\mathcal L}]$. For $n>l$ algorithm
$\mathcal{B}$ makes only use of the rules of inference 
(3.13)(a)-(d) in \cite{Ku} in order to derive $F_n$.
\item[$\bullet$]  First of all we start with algorithm $\mathcal{B}$
and prescribe the formulas $F_1; F_2; \ldots;F_l$ in the proof
$[\Lambda]$. Next we extend $[\Lambda]$ to a proof $[\Lambda_0]$
by applying only Rule (c)
a finite number of times in order to substitute all variables
$x_{i+1},\ldots,x_m$ by $\tilde{\lambda}_{i+1},\ldots,\tilde{\lambda}_{m}$
in the formulas $F' \in [\Lambda]$ 
for all R-formulas $F \in B_S$ 
for which $p$ occurs $i-$ary in the R-conclusion of $F$.
After the construction of $[\Lambda_0]$ we pause $\mathcal{B}$ 
and start $\mathcal{A}$.
\item[$\bullet$] Each time when $\mathcal{A}$ has generated
a prime R-formula $R_{\kappa}$ (including equations and with or without variables) we pause algorithm $\mathcal{A}$ and activate 
algorithm $\mathcal{B}$ to generate $R_{\kappa}$ as well
in the sequence of formulas $F_1; F_2; F_3; \ldots$.
Moreover, if $R_{\kappa} = p \, \lambda_1,...,\lambda_i$
with lists $\lambda_1,...,\lambda_i \in {\mathcal L}$,
then $\mathcal{B}$ will also generate the formula 
$\tilde{G}(\lambda_1,...,\lambda_i)$ in a finite number of steps.
Afterwards we pause algorithm $\mathcal{B}$ and activate 
algorithm $\mathcal{A}$ again, and so on.
\end{itemize}

\noindent
It is clear that any R-derivation in $[S;{\mathcal L}]$
can also be performed in $[M;{\mathcal L}]$.
To prove that algorithm $\mathcal{B}$ is well defined we have
to show that it is able to generate the formulas 
$\tilde{G}(\lambda_1,...,\lambda_i)$ once algorithm $\mathcal{A}$ 
has produced the next prime formula of the form
$p \, \lambda_1,...,\lambda_i$. This will be explained now.

\noindent
Let $F$ be any R-formula in $[S;{\mathcal L}]$ 
and suppose that $x_1,\ldots,x_i$ and 
the variables of $G$ are not involved in $F$.
To $F$ there corresponds exactly one formula 
$\hat{F} \in {\mathcal F}$ of $[M;{\mathcal L}]$ 
which results if we replace in $F$ each 
$i-$ary subformula $p \, \lambda_1,...,\lambda_i$
with $\lambda_1,...,\lambda_i \in {\mathcal L}$ by the formula 
$\tilde{G}(\lambda_1,...,\lambda_i)$.
\begin{footnote}
{We have $\hat{F}=F$ if $p$ does not occur $i$-ary in $F$.}
\end{footnote}
We have assumed that the variables of $G$ are not involved in $B_S$.
Then we obtain that $\hat{F}$ is a step of the extended proof 
$[\Lambda_0]$ for all R-formulas $F \in B_S$ 
for which $p$ occurs $i-$ary in the R-conclusion of $F$.

\noindent
Beside the axioms $F$ in $B_S$ for which $p$ occurs $i-$ary in the R-conclusion of $F$ algorithm $\mathcal{A}$ can also make use of the R-axioms of equa\-lity (1.9)(c) with $n=i$ in order to deduce prime R-formulas
$p\,\lambda_1,...,\lambda_i$ in $[S;{\mathcal L}]$ 
from equations and these R-axioms. Let 
\begin{equation*}
F = ~\to ~\sim y_1,y'_1 \ldots \to ~\sim y_i,y'_i
\to ~p\,y_1,\ldots,y_i \, ~p\,y'_1,\ldots,y'_i
\end{equation*}
be such an R-axiom of equality with variables 
$y_k,y'_k \in X$. We suppose that $x_1,...,x_i$ 
and the variables of $G$ are not involved in $F$.
Then we infer in $[M;{\mathcal L}]$ the formula
\begin{equation*}
\begin{split}
\qquad \hat{F} = ~\to ~\sim y_1,y'_1 \ldots \to ~\sim y_i,y'_i
\to ~\tilde{G}(y_1,\ldots,y_i) \, \tilde{G}(y'_1,\ldots,y'_i)\,.
\end{split}
\end{equation*}
That this is possible can be seen by using \cite[(4.9) Corollary]{Ku} combined with
the Deduction Theorem \cite[(4.5)]{Ku}, by using the axioms of equa\-lity and the Equivalence Theorem \cite[Theorem (3.17)(a)]{Ku}.
This will only require the use of the Rules (a)-(d).

\noindent
Let $R_{\lambda}$ be any R-axiom 
generated by algorithm $\mathcal{A}$ and assume
that $p$ occurs $i$-ary in the R-conclusion of $R_{\lambda}$.
Then we summarize and keep in mind that 
we can derive the corresponding formula $\hat{R}_{\lambda}$
in $[M;{\mathcal L}]$ from $[\Lambda]$ by using only Rules (a)-(d).

\noindent
Initially $\mathcal{B}$ generates $[\Lambda_0]$.
We consider a finite part $R_1; \ldots; R_{\alpha}$
of the R-formulas from $\Pi_R(S;{\mathcal L})$ 
generated by the algorithm $\mathcal{A}$,
and we assume that $R_{\alpha}$ is a prime R-formula.
Then we activate algorithm $\mathcal{B}$ and 
proceed with a further expansion 
$F_1; \ldots; F_{\beta}$ of the list of formulas
from $\Pi(M;{\mathcal L})$
until we have derived $R_{\alpha}$ and $F_{\beta}=\hat{R}_{\alpha}$.
This can be achieved if $\mathcal{B}$ mimics the R-derivation
$R_1; \ldots; R_{\alpha}$ in the following way:
\begin{itemize}
\item[$\bullet$] For any R-axiom $R_{\lambda}$ with $\lambda \leq \alpha$
algorithm $\mathcal{B}$ generates the formula
$R_{\lambda}$ as well. If $p$ occurs $i$-ary in the R-conclusion 
of $R_{\lambda}$ then $\mathcal{B}$ generates $\hat{R}_{\lambda}$.
\item[$\bullet$] Suppose that the R-formula $R_{\lambda}$
was derived from the prime R-formula $R_{\kappa}$ and
the R-formula $\to R_{\kappa} R_{\lambda}$ with Rule (b),
$\kappa < \lambda \leq \alpha$.
Then algorithm $\mathcal{B}$ derives the formula $R_{\lambda}$ as well. 
Suppose that $\hat{R}_{\kappa}$ and
$\to \hat{R}_{\kappa} \hat{R}_{\lambda}$ were already derived
by $\mathcal{B}$, but not $\hat{R}_{\lambda}$. Then $\mathcal{B}$ generates 
$\hat{R}_{\lambda}$ from Rule (b).
\item[$\bullet$]  Assume that the R-formula 
$R_{\lambda}=R_{\kappa}\frac{\nu}{x}$ with $x \in X$ and
$\nu \in {\mathcal L}$
was derived from the R-formula $R_{\kappa}$ with Rule (c),
$\kappa < \lambda \leq \alpha$.
Then algorithm $\mathcal{B}$ derives 
the formula $R_{\lambda}$ as well. Suppose that $\hat{R}_{\kappa}$ was already derived by $\mathcal{B}$, but not $\hat{R}_{\lambda}$.
Then $\mathcal{B}$ generates 
$\hat{R}_{\lambda}=\hat{R}_{\kappa}\frac{\nu}{x}$ from Rule (c).
\end{itemize}
\end{itemize}
\end{proof}
\noindent
Under the mild additional condition ${\mathcal L}_* \neq \emptyset$ we have proved a slightly more general version of \cite[(5.4) Conjecture]{Ku},
namely the following theorem, which makes use of Definition \ref{sl}:\\

\begin{thm}\label{conjecture}
Let $M=[S;A_M;P_M;B_M]$ be a mathematical system with an underlying recursive
system $S=[A_S;P_S;B_S]$ such that $A_M = A_S$, $P_M = P_S$, $B_M = B_S$. 
Suppose that $[M;{\mathcal L}]$ is a mathematical system 
with restricted argument lists in ${\mathcal L}$
and that ${\mathcal L}$ is enumerable\,.
Let ${\mathcal L}_* \neq \emptyset$ be the set of all $A_S$-lists 
in ${\mathcal L}$ without variables.
\begin{itemize}
\item[i.] Let be $\lambda, \mu \in {\mathcal L}_*$. Then
$$\sim\,\lambda, \mu \in \Pi(M;{\mathcal L}) ~ \Leftrightarrow ~
\sim\,\lambda, \mu \in \Pi_R(S;{\mathcal L})\,.$$
\item[ii.] Let $p \in P_S$
and $\lambda_1,...,\lambda_i \in {\mathcal L}_*$ for $i \geq 0$ be elementary $A_S$-lists. Then
$$p\,\lambda_1,...,\lambda_i \in \Pi(M;{\mathcal L})~ \Leftrightarrow ~
p\,\lambda_1,...,\lambda_i \in \Pi_R(S;{\mathcal L})\,.$$
\end{itemize}
\end{thm}
\begin{rem} We have assumed that 
${\mathcal L}_* \neq \emptyset$
in order to avoid trouble with the definition of the interpretation $V$
of the formulas $F \in {\mathcal F}_*$.
\end{rem}
From \cite[Section 5.3]{Ku} and Theorem \ref{conjecture}
we obtain a consistency proof for the following Peano arithmetic PA:
\noindent
Let $\tilde{S}$ be the recursive system $\tilde{S}=[\tilde{A};\tilde{P};\tilde{B}]$ 
where $\tilde{A}$, $\tilde{P}$ and $\tilde{B}$ are empty, 
and introduce the alphabets 
$A_{PA} = [\,0;\,s\,;\,+\,;\,*\,]$, $P_{PA} = [\,]$. 
We define the set ${\mathcal L}$ of
\underline{numeral terms} by the recursive definition

\begin{tabular}{llll}
(i)& $0$ and $x$ are numeral terms for any $x \in X$.&&\\
(ii)& If $\vartheta$ is a numeral term, then also $s(\vartheta)$.&&\\
(iii)& If $\vartheta_1$, $\vartheta_2$ are numeral terms,
then also $+(\vartheta_1 \vartheta_2)$ and $*(\vartheta_1 \vartheta_2)$.&&\\
\end{tabular}\\

We define the mathematical system $M'=[\tilde{S};A_{PA};P_{PA};B_{PA}]$ 
by giving the following basis axioms for $B_{PA}$
with distinct variables $x,y$

\begin{tabular}{llll}
($\alpha$)  & $\forall\,x ~ \sim +(0x),x$&&\\
($\beta$)  & $\forall\,x \, \forall\,y ~\sim +(s(x)y),s(+(xy))$&&\\
($\gamma$)  & $\forall\,x  ~ \sim *(0x),0$&&\\
($\delta$)  & $\forall\,x \, \forall\,y ~ \sim *(s(x)y),+(*(xy)y)$&&\\
($\varepsilon$) & $\forall\,x \, \forall\,y ~ \to ~ \sim s(x),s(y)\,~ \sim x,y$&&\\
($\zeta$) & $\forall x  ~\neg \sim s(x),0\,.$&&\\
\end{tabular}\\

\noindent
Moreover, for all formulas $F$ (with respect to $A_{PA}$ and $P_{PA}$)
which have only numeral argument lists, the following formulas
belong to $B_{PA}$ according to the Induction Scheme

\begin{tabular}{llll}
(IS) & $\to ~~ \forall\,x\,~ \& ~ \mbox{SbF}(F;0\,;x) ~\, 
   \to~ F\, \mbox{SbF}(F;s(x)\,;x)  ~~
   \forall\,x\,F$ \,.&&\\ 
\end{tabular}

\noindent
The system PA of \underline{Peano arithmetic} is given by
PA = $[M';{\mathcal L}]$, i.e. the argument lists of PA are restricted 
to the set ${\mathcal L}$ of numerals.
The Induction Rule (3.13)(e) is not used in PA since 
$\tilde{A}$, $\tilde{P}$ and $\tilde{B}$ are empty here
and since we are using the Induction Scheme (IS).
Let us define a recursive system $S=[A_S;P_S;B_S]$ as follows:

\noindent
We choose $A_S=A_{PA}=[\,0;\,s\,;\,+\,;\,*\,]$, $P_S=[\,N_0\,]$ and $B_S$ 
consisting of the basis R-axioms with distinct variables $x,y$\\

\begin{tabular}{llll}
(1)  & $N_0\,0$&&\\
(2)  & $\to ~ N_0\,x \,~ N_0\,s(x)$&&\\
(3)  & $\to ~ N_0\,x ~ \sim +(0x),x$&&\\
(4)  & $\to ~ N_0\,x ~ \to ~ N_0\,y\,~ \sim +(s(x)y),s(+(xy))$&&\\
(5)  & $\to ~ N_0\,x ~ \sim *(0x),0$&&\\
(6)  & $\to ~ N_0\,x ~ \to ~ N_0\,y\,~ \sim *(s(x)y),+(*(xy)y)$&&\\
(7) & $\to ~ N_0\,x ~ \to ~ N_0\,y ~ \to ~ \sim s(x),s(y)\,~ \sim x,y$\,.&&\\
\end{tabular}\\

\noindent
Using the results in \cite[Chapter (5.3)]{Ku} and Theorem \ref{conjecture} 
the inconsistency of PA would imply that there is
an elementary numeral term $\lambda$, i.e. a numeral term without variables, 
such that $N_0\,\lambda$ as well as $\sim\,s(\lambda),0$ are $R$-derivable in $S$, which is impossible.\\
 
\noindent
Instead of explaining these earlier results in detail again we will now derive a stronger result from Theorem \ref{mainthm}, namely 
%
\begin{thm}{\bf The $\omega$-consistency of the Peano arithmetic PA}
\label{pa_omega_consistency}\\
\noindent
Let $F$ be a formula in PA with $\mbox{free}(F)=\{x\}$, i.e. let $x \in X$ be the only free variable of $F$. Suppose that $\neg F \frac{\lambda}{x}$ is provable in PA for all elementary 
numeral terms $\lambda$. Then $\exists x F$ is not provable in PA.
\end{thm}
\begin{proof}
We assume that $\exists x F$ is provable in PA, i.e.
\begin{equation}\label{exf}
\exists x F \in  \Pi(PA)\,,
\end{equation}
and will show that this leads to a contradiction. 

\noindent
Step 1. We make use of the recursive system $S=[A_S;P_S;B_S]$ with the basis R-axioms 
(1)-(7) given above and define the mathematical system $M=[S;A_S;P_S;B_S]$. Recall the set ${\mathcal L}$ of numeral terms. Now $[M; {\mathcal L}]$ satisfies the conditions 
of Theorem \eqref{mainthm} mentioned at the beginning of this section.
We will show for \textit{all} formulas $H$ in $[M;{\mathcal L}]$ that 
\begin{equation}\label{inductions}
\to ~~ \forall\,x\,\to~ N_0\,x ~~ \& ~ H\frac{0}{x} ~ 
   \to\,H ~\,  H\frac{s(x)}{x} \quad
   \forall\,x\,\to~ N_0\,x ~ H
\end{equation}
is provable in $[M;{\mathcal L}]$, which 
is the induction principle for $[M;{\mathcal L}]$.
Without loss of generality we may assume that $\mbox{free}(H)=\{x, x_1,\ldots,x_m\}$
with disjoint variables $x, x_1,\ldots,x_m$ and $m \geq 0$.
Corresponding to the variables $x_1,\ldots,x_m$ we choose new 
and different constant symbols $c_1,\ldots,c_m$ and put
$$
\tilde{H}=H\frac{c_1}{x_1}\ldots\frac{c_m}{x_m}\,.
$$ 
We define $A =A_S \cup \{c_1,\ldots,c_m\}$, 
\[ \mathcal{L}_A = \{\, \lambda \frac{c_1}{y_1}...\frac{c_m}{y_m} \,\, : \,\, 
\lambda \in \mathcal{L}\,, y_1,\ldots,y_m \in X\,\}
\]
and the extension $M_A = [S; A; P_S; B_S]$.
Due to \cite[Definition (4.2)(d)]{Ku} and \cite[Corollary (4.9)(a)]{Ku}
we obtain a mathematical system $[M_A;\mathcal{L}_A]$ with argument lists
restricted to $\mathcal{L}_A$. We adjoin the statement
$$
\varphi =\forall\,x\,\to~ N_0\,x ~~ \& ~ \tilde{H}\frac{0}{x} ~ 
   \to\,\tilde{H} ~\,  \tilde{H}\frac{s(x)}{x}  
$$
to $[M_A;\mathcal{L}_A]$ and obtain the 
extended system $[M_A(\varphi);{\mathcal L}_A]$, see \cite[Definition (4.2)(a)]{Ku}. 
Now $\varphi$ is provable in $[M_A(\varphi);{\mathcal L}_A]$, and we obtain from
the quantifier axiom (3.11)(a) and the Modus Ponens Rule in \cite{Ku} that
\begin{equation}\label{varphi_ohne_all}
\to~ N_0\,x ~~ \& ~ \tilde{H}\frac{0}{x} ~ 
   \to\,\tilde{H} ~\,  \tilde{H}\frac{s(x)}{x} \in \Pi(M_A(\varphi);{\mathcal L}_A)\,.
\end{equation}
Let $u \in X$  be any variable which is neither involved in $H$ nor in $B_S$ and put	
\begin{equation}\label{Gabbr}
\tilde{G} = \& \, N_0 \, u ~ \tilde{H}\frac{u}{x} \,.
\end{equation}
In $[M_A(\varphi);{\mathcal L}_A]$ we obtain a proof 
containing the formula in \eqref{varphi_ohne_all}:
\begin{equation*}
\begin{split}
[\ldots;
& \to~ N_0\,x ~~ \& ~ \tilde{H}\frac{0}{x} ~ 
   \to\,\tilde{H} ~\,  \tilde{H}\frac{s(x)}{x}; \\
	&\to \quad \to~ N_0\,x ~~ \& ~ \tilde{H}\frac{0}{x} ~ 
   \to\,\tilde{H} ~\,  \tilde{H}\frac{s(x)}{x} ~ \to~ N_0\,x ~\tilde{H}\frac{0}{x}; \\
		&\to~ N_0\,x ~\tilde{H}\frac{0}{x}; ~
		\to~ N_0\,0 ~\tilde{H}\frac{0}{x}; ~
  N_0\,0; ~
\tilde{H}\frac{0}{x}; \\
& \to N_0\,0 ~ \to \tilde{H}\frac{0}{x} ~\& \,N_0 \, 0 ~ \tilde{H}\frac{0}{x}; \\
& \to \tilde{H}\frac{0}{x} ~ \& \, N_0 \, 0 ~ \tilde{H}\frac{0}{x}; \\
& \& \, N_0 \, 0 ~ \tilde{H}\frac{0}{x}; \\
& \to ~ N_0\,x ~ N_0 \,s(x);\\
&\to \quad \to~ N_0\,x ~~ \& \, \tilde{H}\frac{0}{x} ~ 
   \to\,\tilde{H} ~\,  \tilde{H}\frac{s(x)}{x} \\
& \to \quad \to ~ N_0\,x ~ N_0 \,s(x)\\
& \quad \quad \to \quad \& \, N_0 \, x ~ \tilde{H} \quad 
\& \, N_0 \, s(x) ~ \tilde{H}\frac{s(x)}{x};\\
& \to \quad \to ~ N_0\,x ~ N_0 \,s(x)\\
& \quad \quad \to \quad \& \, N_0 \, x ~ \tilde{H} \quad 
\& \, N_0 \, s(x) ~ \tilde{H}\frac{s(x)}{x};\\
& \quad \quad \to \quad \& \, N_0 \, x ~ \tilde{H} \quad 
\& \, N_0 \, s(x) ~ \tilde{H}\frac{s(x)}{x};\\
& \to \quad N_0\,u ~\tilde{G}\,]\,.\\
	\end{split}
\end{equation*}
The last step results from the Induction Rule (e), using the abbreviation $\tilde{G}$
in \eqref{Gabbr}. We see that the two formulas
$$
\to~ N_0\,x ~ \tilde{H} \quad \mbox{~and~} \quad \forall\,x\,\to~ N_0\,x ~ \tilde{H}
$$
are also provable in $[M_A(\varphi);{\mathcal L}_A]$. 
It follows from the Deduction Theorem \cite[(4.3)]{Ku} that the formula
$$
\to ~~ \forall\,x\,\to~ N_0\,x ~~ \& ~ \tilde{H}\frac{0}{x} ~ 
   \to\,\tilde{H} ~\,  \tilde{H}\frac{s(x)}{x} \quad
   \forall\,x\,\to~ N_0\,x ~ \tilde{H}
$$
is provable in $[M_A;{\mathcal L}_A]$. From the generalization of the 
constant symbols $c_1,\ldots c_m$ according to \cite[Corollary (4.9)(b)]{Ku} 
we see that the formula \eqref{inductions}
is provable in the original mathematical system $[M;{\mathcal L}]$.

\noindent
Step 2: Following \cite[Section 5]{Ku} we construct from PA a
related mathe\-matical system $PA_{N_0}=[M_{PA_{N_0}};\mathcal{L}]$ with argument lists restricted to the numerals $\mathcal{L}$ as follows:
We put $M_{PA_{N_0}}=[\tilde{S};A_S;P_S;B_{PA_{N_0}}]$ with the underlying recursive system 
$\tilde{S}=[\,[\,];[\,];[\,]\,]$, and recall that $A_S=[\,0;\,s\,;\,+\,;\,*\,]$, $P_S=[N_0]$.
The basis axioms $B_{PA_{N_0}}$ of PA$_{N_0}$ are given by the two formulas 
$N_0\,0$ and $\to~N_0\,x\,N_0\,s(x)$ with $x \in X$
and by all the formulas $\Gamma_{N_0}(G)\,\Psi_{N_0}(G)$, where G is any basis axiom of PA
(including the formulas from the induction scheme). 
Here $\Gamma_{N_0}(G)$ and $\Psi_{N_0}(G)$ are defined in \cite[Section 5]{Ku}
for every PA-formula $G$ as follows:
\begin{itemize}
\item We put $\Gamma_{N_0}(G)=~\to~N_0\,x_1~...~\to~N_0\,x_n$
for the block of $N_0$-premises with respect to $\mbox{free}(G)=\{x_1;\ldots;x_n\}$,
$x_1,\ldots,x_n$ ordered according to their first occurrence in $G$. 
For $n=0$ the string $\Gamma_{N_0}(G)$ is defined to be empty.
\item $\Psi_{N_0}(G)$ results from $G$ if we replace simultaneously in every subformula
$\forall z G'$ of $G$ the part $\forall z$ by $\forall z \rightarrow N_0 z$,
and in every subformula
$\exists z G'$ of $G$ the part $\exists z$ by $\exists z ~ \& ~ N_0 z$, with $z \in X$. 
\end{itemize}
Lemma (5.2)(iii) in \cite[Section 5]{Ku} states that $\Gamma_{N_0}(G)\,\Psi_{N_0}(G)$
is provable in PA$_{N_0}$ for every formula $G$ which is provable in PA.\\
From $\neg F\frac{\lambda}{x} \in \Pi(PA)$ for all 
$\lambda \in {\mathcal L}_*$ and from \eqref{exf} we see that
\begin{equation}\label{negpano}
\neg \Psi_{N_0}\left(F\frac{\lambda}{x}\right)= \neg \Psi_{N_0}(F)\frac{\lambda}{x}\in \Pi(PA_{N_0})\,,
\end{equation}
\begin{equation}\label{expano}
\exists x \, \& N_0\, x \, \Psi_{N_0}(F) \in \Pi(PA_{N_0})\,.
\end{equation}

\noindent
Step 3. To the mathematical system $[M;{\mathcal L}]$ we adjoin the single statement

\begin{tabular}{llll}
$(*)$ \qquad && $\forall x \, \to ~ N_0\,x ~\neg \sim s(x),0$&\\
\end{tabular}

\noindent
and obtain the mathematical system $M_{PA}=[M((*));{\mathcal L}]$.
We see from the first step that every formula which is provable in PA$_{N_0}$
is also provable in $M_{PA}$. It follows from \eqref{negpano}, \eqref{expano}
and the Deduction Theorem \cite[(4.3)]{Ku} for \textit{all} $\lambda \in {\mathcal L}_*$ that
\begin{equation}\label{negpanoml}
\to (*) ~ \neg \Psi_{N_0}(F)\frac{\lambda}{x}
\in \Pi(M;{\mathcal L})\,,
\end{equation}
\begin{equation}\label{expanoml}
\to (*) ~ \exists x \, \& N_0\, x \, \Psi_{N_0}(F) \in \Pi(M;{\mathcal L})\,.
\end{equation}

\noindent
Step 4. We extend the function $V$ to the set of all statements in $[M; {\mathcal L}]$.
Let $F_R$ be an elementary prime formula in $[M; {\mathcal L}]$.
Then $F_R$ is also an elementary prime R-formula in $[S; {\mathcal L}]$.
In this case we recall that $V(F_R)=\{\top\}$ iff $F_R$ is R-derivable in $[S; {\mathcal L}]$,
i.e. iff $F_R \in \Pi_R(S; {\mathcal L})$, and otherwise we have $V(F_R)=\emptyset$.
We put in addition for all statements $G, H$ of $[M;{\mathcal L}]$:
\begin{equation*}
\begin{split}
V(\neg G)&=\{\top\} \setminus V(G)\,,\\
V(\rightarrow G H) &= (\{\top\} \setminus V(G)) \cup V(H)\,,\\
V(\vee G H) &= V(G) \cup V(H)\,, \\
V(\& G H) &= \,V(G) \cap V(H)\,,\\
V(\leftrightarrow G H) &= V(\rightarrow G H) \cap V(\rightarrow H G)\,.\\
\end{split}
\end{equation*}
Recall the set ${\mathcal L}_*$ of all elementary numeral terms (without variables)
and let $G$ be a formula of $[M;{\mathcal L}]$ with $\mbox{free}(G) \subseteq \{z\}$.
Then $G\frac{\lambda}{z}$ is a statement for all $\lambda \in {\mathcal L}_*$, and we put
\begin{equation*}
\begin{split}
V(\forall\,z\, G) &=\bigcap \limits_{\lambda \in {\mathcal L}_*} 
V\left(G\frac{\lambda}{z}\right)\,,\\
V(\exists\,z\, G) &=\bigcup \limits_{\lambda \in {\mathcal L}_*} 
V\left(G\frac{\lambda}{z}\right)\,.\\
\end{split}
\end{equation*}
Recalling the definition of the function $\Theta$ 
we can use induction over well formed formulas and the Equivalence Theorem 
\cite[(3.17)(a) Theorem]{Ku} and obtain for all formulas $G$ of $[M;{\mathcal L}]$: 
\begin{equation}\label{theta_equiv}
\leftrightarrow G ~ \Theta(G) \in \Pi(M;{\mathcal L})\,.
\end{equation}
Next we define a degree for all formulas in $[M;{\mathcal L}]$.
We put $\mbox{deg}(G)=0$ for all prime formulas $G$.
For general formulas $G$, $H$ in $[M;{\mathcal L}]$ we put
$\mbox{deg}(\neg G)=\mbox{deg}(G)+1$, 
$$\mbox{deg}(J\, G \, H)=\max(\mbox{deg}(G),\mbox{deg}(H))+1 \,\mbox{~for~} \,
J \in \{\rightarrow; \vee; \& ;\leftrightarrow\}\,,$$
and for $z \in X$ we put
$$\mbox{deg}(\forall\,z\, G)=\mbox{deg}(\exists\,z\, G)=\mbox{deg}(G)+1 \,.$$
Induction over $n \in \setN_0$ with respect to $\mbox{deg}(G)\leq n$ 
gives for all \textit{statements} $G$ that
\begin{equation}\label{vstate}
V(\Theta(G))=V(G)\,.
\end{equation}
\noindent
Step 5. In the final step we apply Theorem \ref{mainthm},
\eqref{theta_equiv}, \eqref{vstate}
on the mathematical system $[M;{\mathcal L}]$ with the underlying recursive system $S$
and on the two statements in \eqref{negpanoml}, \eqref{expanoml}. We obtain
for all $\lambda \in {\mathcal L}_*$ that
\begin{equation}\label{negfin}
V \left(\to (*) ~ \neg \Psi_{N_0}(F)\frac{\lambda}{x}\right)=\{\top\}\,,
\end{equation}
\begin{equation}\label{exfin}
V \left(\to (*) ~ \exists x \, \& N_0\, x \, \Psi_{N_0}(F)\right)=\{\top\}\,.
\end{equation}
Using $N_0\,\lambda \in \Pi_R(S;{\mathcal L})$, 
$\sim s(\lambda),0 \notin \Pi_R(S;{\mathcal L})$ for all $\lambda \in {\mathcal L}_*$
we see $V \left((*)\right)=\{\top\}$\,.
We obtain from \eqref{exfin}
$$
\{\top\} = V  \left(\exists x \, \& N_0\, x \, \Psi_{N_0}(F)\right) 
=\bigcup \limits_{\lambda \in {\mathcal L}_*} 
V\left(\Psi_{N_0}(F)\frac{\lambda}{x}\right)\,.
$$
But \eqref{negfin} gives 
$\begin{displaystyle}
\emptyset =
V\left(\Psi_{N_0}(F)\frac{\lambda}{x}\right)
\end{displaystyle}$
for all $\lambda \in {\mathcal L}_*$\,, a contradiction.
\end{proof}

\subsection{A further example with formal induction}\label{question}\hfill\\
Finally we go back to the recursive system $S=[A;P;B]$
introduced in Section \ref{dual_recursive} with $A=[a;0;1]$, $P=[D]$ 
and the set $B$ consisting of the six basis R-axioms $(\alpha)$-$(\zeta)$. 
Let ${\mathcal L}$ be the set generated by the rules
\begin{itemize}
\item[(i)] $x \in \mathcal{L}$ for all $x \in X$,
\item[(ii)] $0 \in \mathcal{L}$, $1 \in \mathcal{L}$ and $a \in \mathcal{L}$,
\item[(iii)] If $\lambda, \mu \in \mathcal{L}$ then 
$\lambda \mu \in \mathcal{L}$.
\end{itemize} 
We define the mathematical system
$[M;{\mathcal L}]$ with $M=[S;A;P;B]$ and will show that the formula
$\begin{displaystyle}
\forall x \, \leftrightarrow \, D \, x \, \exists \,y\,D\,x,y
\end{displaystyle}$
is provable in $[M;{\mathcal L}]$.
We will present a short semi-formal proof. Due to Rule (d)
it is sufficient to show that
$\begin{displaystyle}
\leftrightarrow \, D \, x \, \exists \,y\,D\,x,y
\end{displaystyle}$
is provable in $[M;{\mathcal L}]$.
For this purpose we will apply the Induction Rule (e) twice to deduce 
$\begin{displaystyle}
\rightarrow \, \exists \,y\,D\,x,y ~ D \, x 
\end{displaystyle}$ as well
$\begin{displaystyle}
\rightarrow \, D \, x \, \exists \,y\,D\,x,y
\end{displaystyle}$
in $[M;{\mathcal L}]$.
Let $x,y,u,v \in X$ be distinct.
Due to Rule (a) the $R$-axioms in $B$ are provable in $[M;{\mathcal L}]$:
\begin{itemize}
\item[1.] $D\, 1$
\item[2.] $\to ~D\, x ~ D\, x0$
\item[3.] $\to ~D\, x ~ D\, x1$
\item[4.] $D\, 1,a$
\item[5.] $\to ~D\, x,y ~ D\, x0,yy$
\item[6.] $\to ~D\, x,y ~ D\, x1,yya$
\end{itemize}
For the first application of Rule (e) we put $p = D$, $i=2$, $x_1=u$, $x_2=v$ and
$G = D\,u$. In axioms 4.-6. we replace the prime subformulas $D\,\lambda_1,\lambda_2$
by $D\,\lambda_1$ and obtain axioms 1.-3. for the 1-ary predicate ``$D$".
Due to Rule (e) 
\begin{itemize}
\item[7.] $\to ~D\, u,v ~ D\, u$
\end{itemize}
is provable in $[M;{\mathcal L}]$, and also the formulas

\begin{itemize}
\item[8.] $\to ~\neg \, D\, x ~\, \neg \,D\, x,y$
\item[9.] $\forall y \, \to ~\neg \, D\, x ~\, \neg \,D\, x,y$ \quad from 8. $\&$ Rule (d)
\item[10.] $\to ~\neg \, D\, x ~\, \forall y \, \neg \,D\, x,y$ \quad with 9. $\&$ 
quantifier axiom (3.11)(b)
\item[11.] $\to  ~ \neg\, \forall y \, \neg \,D\, x,y ~\, D\, x$
\item[12.] $\to  ~\exists y \, D\, x,y ~\, D\, x$ \quad with 11. $\&$ quantifier axiom (3.11)(c)
\end{itemize}
This is the first implication. 

For the second one we deduce the following formulas in
$[M;{\mathcal L}]$:

\begin{itemize}
\item[13.] $\to ~\,D\, x,y ~\exists y \, D\, x,y$ \, from example 1 in Section \ref{general_valid}
\item[14.] $\to ~\,D\, 1,a ~\exists y \, D\, 1,y$ \quad from 13. and two times Rule (c)
\item[15.] $\exists y \, D\, 1,y$ \quad with 4. and 14.
\item[16.]  $\to ~\,D\, x0,y ~\exists y \, D\, x0,y$  \, from example 1 in Section \ref{general_valid}
\item[17.]  $\to ~\,D\, x0,yy ~\exists y \, D\, x0,y$ \,  from 16. and Rule (c)
\item[18.]  $\to ~\,D\, x,y ~\, \exists y \, D\, x0,y$ \, with 5. and 17.
\item[19.]  $\to ~\,\neg\,\exists y \, D\, x0,y ~\, \neg\, D\, x,y $  
\item[20.]  $\forall y\, \to ~\,\neg\,\exists y \, D\, x0,y ~\, \neg\, D\, x,y$ \, 
from 19. $\&$ Rule (d)
\item[21.]  $\to ~\,\neg\,\exists y \, D\, x0,y ~ \forall y\, \neg\, D\, x,y$  with 20. $\&$
quantifier axiom (3.11)(b)
\item[22.]  $\to ~\,\neg\,\forall y\, \neg\, D\, x,y ~\exists y \, D\, x0,y $  
\item[23.]  $\to ~\exists y\, D\, x,y ~\exists y \, D\, x0,y$ \, 
with 22. $\&$ quantifier axiom (3.11)(c)
\item[24.]  $\to ~\,D\, x1,y ~\exists y \, D\, x1,y$ \, from example 1 in Section \ref{general_valid}
\item[25.]  $\to ~\,D\, x1,yya ~\exists y \, D\, x1,y$ \, from 24. and Rule (c)
\item[26.]  $\to ~\,D\, x,y ~\, \exists y \, D\, x1,y$ \, with 6. and 25.
\item[27.]  $\to ~\,\neg\,\exists y \, D\, x1,y ~\, \neg\, D\, x,y $  
\item[28.]  $\forall y\, \to ~\,\neg\,\exists y \, D\, x1,y ~\, \neg\, D\, x,y$ \, 
from 27. $\&$ Rule (d)
\item[29.]  $\to ~\,\neg\,\exists y \, D\, x1,y ~ \forall y\, \neg\, D\, x,y$  with 28. $\&$
quantifier axiom (3.11)(b)
\item[30.]  $\to ~\,\neg\,\forall y\, \neg\, D\, x,y ~\exists y \, D\, x1,y $  
\item[31.]  $\to ~\exists y\, D\, x,y ~\exists y \, D\, x1,y$ \, 
with 30. $\&$ quantifier axiom (3.11)(c)
\end{itemize}
From formulas 15. 23., 31. and \cite[Theorem (3.17)(b)]{Ku}
we obtain that the formulas
\begin{itemize}
\item[32.]  $\exists v\, D\, 1,v$ 
\item[33.]  $\to ~\exists v\, D\, x,v ~\exists v \, D\, x0,v$ 
\item[34.]  $\to ~\exists v\, D\, x,v ~\exists v \, D\, x1,v$ 
\end{itemize}
are provable in $[M;{\mathcal L}]$.

\noindent
For the second application of Rule (e) we put $p = D$, $i=1$, $x_1=u$ and
$G = \exists v \, D\, u,v$. We replace the prime subformulas $D\,\lambda_1$
in axioms 1.-3. by $\exists v \, D\, \lambda_1,v$ 
and obtain formulas 32.-34., respectively.

\noindent
Due to Rule (e) we see that $\to \, D\,u \, \exists v\, D\, u,v$ and hence
$\begin{displaystyle}
\rightarrow \, D \, x \, \exists \,y\,D\,x,y
\end{displaystyle}$
are both provable in $[M;{\mathcal L}]$.

\section{Conclusions and outlook}\label{conclusions}
\noindent
We have presented contributions to elementary proof theory.
Especially in Section \ref{prime_formulas} we have determined a simple procedure in order to eli\-mi\-nate prime formulas from formal proofs 
which do not occur with a given arity in the basis axioms of a mathematical system. We also hope to develop a method in order to eliminate equations from formal proofs if there are no equations in the basis axioms.

\noindent
Our most important contribution is Theorem \ref{mainthm}, which is a general result
of mathematical logic concerning formal induction.
We have presented two applications of this theorem in Section \ref{pa}, 
namely the proof of \cite[(5.4) Conjecture]{Ku}, see Theo\-rem
\ref{conjecture}, and the $\omega$-consistency of the Peano arithmetic PA 
in Theorem \ref{pa_omega_consistency}. 

\noindent
It would be very interesting to create a computer program which is able to check semiformal
proofs like in Section \ref{question}. First a machine should be able to check fully formalized proofs
with certain restrictions. For example, the number of propositional variables 
in the axioms \cite[(3.9)]{Ku} must be small enough for an efficient calculation.
In a next step the program should be extended to analyze the use of the axioms and rules in order to develop further composed rules of inference, especially for the propositional calculus
and for the treatment of equations.

\noindent An advanced program should also make use of 
\cite[Theorem (3.17), Propositions (3.18),(3.19), Theorems (4.5),(4.8), Corollaries (4.9),(4.10)]{Ku}.

\end{document}